%% file: Strong_Regret_9_27_2012.tex
\documentclass[11pt, draftclsnofoot, onecolumn]{IEEEtran}
\usepackage{multirow}
\usepackage{amsfonts}%
\usepackage{setspace}
\usepackage{amssymb}%
\usepackage{booktabs,blindtext}
\usepackage[dvips]{graphicx}%
\usepackage{amsmath}%
\usepackage[usenames]{color}
\usepackage{algorithmic}
\usepackage{comment}
\usepackage{enumerate}
\newtheorem{theorem}{Theorem}
\newtheorem{corollary}{Corollary}
\newtheorem{lemma}{Lemma}
\newtheorem{definition}{Definition}
\newtheorem{example}{Example}
\newtheorem{condition}{Condition}
\newtheorem{assumption}{Assumption}

\newtheorem{remark}{Remark}
\DeclareMathOperator*{\argmin}{arg\,min}
\DeclareMathOperator*{\argmax}{arg\,max}

\ifodd 0
\newcommand{\rdec}[1]{{\color{blue}#1}}
  
\else
\newcommand{\rdec}[1]{#1}
  
\fi

\ifodd 0
\newcommand{\revq}[1]{{\color{blue}#1}}
\else
\newcommand{\revq}[1]{#1}
\fi

\ifodd 0
\newcommand{\newt}[1]{{\color{blue}#1}}
\else
\newcommand{\newt}[1]{#1}
\fi

\ifodd 0

\newcommand{\com}[1]{\textbf{\color{red}(COMMENT: #1)}} 
\newcommand{\clar}[1]{\textbf{\color{green}(NEED more work: #1)}}

\else

\newcommand{\com}[1]{}
\newcommand{\clar}[1]{}
\fi

\ifodd 0
\newcommand{\comc}[1]{\textbf{\color{red}(COMMENT: #1)}} 
\else
\newcommand{\comc}[1]{}
\fi

\ifodd 0
\newcommand{\aug}[1]{{\color{blue}#1}}
\else
\newcommand{\aug}[1]{#1}
\fi

\begin{document}
\title{Learning of Uncontrolled Restless Bandits with Logarithmic Strong Regret}
\author{\IEEEauthorblockN{Cem Tekin,~\IEEEmembership{Member,~IEEE}, Mingyan Liu,~\IEEEmembership{Senior Member,~IEEE}\\
\thanks{C. Tekin is with the Electrical Engineering Department, University of California, Los Angeles, CA, USA, cmtkn@ucla.edu. M. Liu is with the Electrical Engineering and Computer Science Department, University of Michigan, Ann Arbor, MI 48105, USA, mingyan@eecs.umich.edu. Part of this work was done when C. Tekin was a graduate student in the Electrical Engineering and Computer Science Department, University of Michigan.
A preliminary version of this work appeared in Allerton 2011. The work is partially supported by the NSF under grants CIF-0910765 and CNS-1217689, and by the ARO under Grant W911NF-11-1-0532.  }
}}

\maketitle

\begin{abstract}
In this paper we consider the problem of learning the optimal \rdec{dynamic} 
policy for uncontrolled restless bandit problems. In an uncontrolled restless bandit problem, there is a finite set of arms, each of which when \rdec{played} yields a \rdec{non-negative} reward. There is a player who sequentially selects one of the arms at each time step. The goal of the player is to maximize its undiscounted reward over a time horizon $T$. The reward process of each arm is a finite state Markov chain, whose transition probabilities are unknown to the player. State transitions of each arm is independent of the player's actions, thus ``uncontrolled''. 
We propose a learning algorithm with \newt{near-logarithmic regret} uniformly over time with respect to the \rdec{optimal (dynamic) finite horizon policy, referred to as strong regret, to contrast with commonly studied notion of weak regret which is with respect to the optimal (static) single-action policy}. \newt{We also show that when an upper bound on a function of the system parameters is known, our learning algorithm achieves logarithmic regret.} 
Our results extend the literature on optimal adaptive learning of Markov Decision Processes (MDPs) to Partially Observed Markov Decision Processes (POMDPs). \newt{Finally, we provide numerical results on a variation of our proposed learning algorithm and compare its performance and running time with other bandit algorithms.}
\end{abstract}


 \input{introduction}
\section{Related Work} \label{single:sec:strong-related}
\input{strong_related}

\section{Problem Formulation and Preliminaries} \label{single:sec:strong-probform}
\subsection{Problem Formulation and Notations} 
 \input{strong_probform}

 \subsection{Solutions to the Average Reward Optimality Equation} \label{single:sec:strong-aroe}
 \input{strong_aroe}
 \subsection{Countable Representation of the Information State} \label{single:sec:strong-countable}
 \input{strong_countable}

 \section{Average Reward with Estimated Probabilities (AREP) Algorithm} \label{single:sec:strong-alg}
 \input{strong_alg}

 \section{Finite Partitions of the Information State} \label{single:sec:defnfinpar}
 \input{strong_finpar}

 \input{equivalentiid}
 \section{Analysis of the Regret of AREP} \label{single:sec:strong-analysis}
 \input{strong_analysis}
 \subsection{An Upper Bound on the Regret}  \label{single:sec:strong-regretbound}
 \input{strong_regretbound}

 \subsection{Bounding the Expected Number of Explorations} \label{single:sec:strong-explorations}
 \input{strong_analysis_explorations}
 \subsection{Bounding $E^{\boldsymbol{P}}_{\psi_0,\alpha}[D_1(\epsilon,J_l,u)]$ for a suboptimal action $u \notin O(J_l; \boldsymbol{P})$} \label{single:sec:strong-para2}
 \input{strong_para2}

 \subsection{Bounding $E^{\boldsymbol{P}}_{\psi_0,\alpha}[D_2(\epsilon, J_l)]$} \label{single:sec:strong-para1}
 \input{strong_para1}
 \subsection{Logarithmic regret upper bound} \label{single:sec:strong-para3}
 \input{strong_para3}
 \section{AREP with an Adaptive Exploration Function} \label{sec:extensions_explore}
 \input{strong_extension_explore}

 \section{AREP with Finite Partitions} \label{sec:extensions_finpar}
 \input{strong_extension_finpar}

%
\input{numerical}

\section{Conclusion} \label{sec:conclusion}
\rdec{We showed for an uncontrolled restless bandit problem there exist online learning algorithms with logarithmic regret uniformly in time with respect to the optimal finite horizon expected total reward policy. This result assumes that the player is able to solve the average reward optimization problem sufficiently often, which can be computationally costly.  In practice, the player can use value iteration and belief state space discretization to obtain an approximate solution.  Furthermore, the player can simply choose to use an explicit/structured policy (e.g., a greedy or one-step lookahead policy) that does not require much computational effort.  In some instances such policies may be optimal or suboptimal with performance guarantees. Provided that the approximate solution holds similar continuity properties, then using AREP (substituting the part solving the AROE with an approximation or an explicit policy) results in a learning algorithm with logarithmic regret w.r.t the chosen substitute.  \newt{We showed via numerical results that the real-time performance of the learning algorithms proposed in this paper is much better than the previous online learning algorithms, and our algorithms can be efficiently implemented using approximation methods.}
}

 \appendices
 \section{Results Regarding Deviations of Estimated Transition Probabilities} \label{s:urestls:app:lemmalargedev1}
 \input{s_urestls_app_lemmalargedev1}

 \section{Proof of Lemma \ref{s:urestls:lemma:largedev}} \label{s:urestls:app:lemmalargedev}
 \input{s_urestls_app_lemmalargedev}

\section{Proof of Lemma \ref{s:urestls:lemma:subsec30}}
 \label{s:urestls:app:lemma:subsec30}
 \input{s_urestls_app_lemma_subsec30}
 \section{Proof of Lemma \ref{s:urestls:lemma:subsec34}} \label{s:urestls:app:subsec34}
\input{s_urestls_app_lemmasubsec34}
\section{Proof of Lemma \ref{s:urestls:lemma:subsec35}}
\label{s:urestls:app:lemma:subsec35}
\input{s_urestls_app_lemma_subsec35}
%
\section{Proof of Lemma \ref{s:urestls:lemma:subsec21}}
\label{s:urestls:app:lemma:subsec21}
\input{s_urestls_app_lemma_subsec21}


\bibliographystyle{IEEEtran}
\bibliography{cemreferences}

\end{document}

%% file: introduction.tex
\section{Introduction} \label{sec:introduction}

In an {\em uncontrolled restless bandit problem} (URBP) there is a
set of arms indexed by $1,2,\ldots,K$, whose state process is
discrete and follows a discrete time Markov rule independent of each other. There is a user/player who
chooses one arm at each of the discrete time steps, gets a reward dependent on the state of the arm, and observes
the current state of the selected arm. The control action, i.e., the arm
selection, does not affect the state transition, therefore the underlying system dynamics is uncontrolled. 
{However, judicious arm selections allow the player to obtain high instantaneous reward (exploitation) and decrease the uncertainty about the current state of the system (exploration), and the key in designing an optimal policy lies in a good balance between exploration and exploitation. } 

If the structure of the system, i.e., the state transition
probabilities and the rewards of the arms are known, then the
optimal policy for an infinite horizon problem can be found by using dynamic programming. 
In the case of infinite horizon with discounted reward, stationary
optimal policies can be found by using contraction properties of the dynamic programming operator.
For the infinite horizon average reward case, stationary optimal policies can be found 
under some assumptions on the transition probabilities \cite{platzman1980, hsu2006}. 

In this paper, rather than finding the optimal policy given the structure of the system (referred to as the optimization problem), we consider the {\em learning} version of the problem, where we assume that initially the player has no knowledge on the transition probabilities of the arms. 
%
This type of learning problem arise in many applications. Examples include sequential channel selection in a multi-channel wireless system where a user initially has no information on the channel statistics, and target tracking where initially the statistics of the target's movement is unknown. 
Our goal is to design learning algorithms with the fastest convergence rate, i.e., the minimum regret, where the {\em regret}  of a learning policy at time $t$ is defined as the difference between the total undiscounted reward of
the optimal dynamic policy for the finite $t$-horizon undiscounted reward problem given full statistics of the system model, and that of the learning policy up to time $t$.  
{It should be noted that this is a form of {\em strong regret}, as the comparison benchmark is the optimal dynamic policy, the best causal policy that can be achieved given full statistics of the system.  By contrast, a much more commonly used performance criterion is the {\em weak regret}, which is the difference between a learning policy and the best {\em single-action policy}, a static policy that always plays the same arm.  Also, to determine the best single-action policy one does not need to know the full statistics of the system, but only the average reward of each arm.  For simplicity throughout the paper, the term {\em regret} refers to the {\em strong regret} defined above. 
} 

In this paper, we show that when the transition probability between any two states of the same arm is always positive, \newt{and when the player knows an upper bound on a function of the system parameters, which we will define later,}
an algorithm with logarithmic regret uniform in time for the finite horizon undiscounted reward problem exists. \newt{If such a bound is not known, we show that near-logarithmic regret can be achieved.}
\newt{We would like to note that the iid bandit problem where the rewards for arm $k$ are drawn from a distribution $f_k$ with finite support (such as the Bernoulli bandit problem) is a special case of the URBP. Since it is proven in \cite{lairobbins1985} that iid bandits have a logarithmic lower bound on the regret, this bound also holds for URBP.}

\newt{To the best of our knowledge, this is the first attempt to extend optimal adaptive learning in Markov decision processes (MDPs) to partially observable Markov decision processes (POMDPs), of which the URBP is a sub-class, where there are an uncountably infinite number of information states.
A parallel work \cite{ortner2012regret} considered the idea of state aggregation to solve URBP, and proved a $O(\sqrt{T})$ lower bound on the regret. This bound does not contradict our $O(\log T)$ bound since it is derived for an adversarially selected problem instance. In contrast, the problem instance (including the states, rewards and transition probability matrices of the arms) is fixed in our case, and the constant in the regret that multiplies the logarithmic time order depends on these parameters. Therefore our $O(\log T)$ bound is instance dependent.  Similarly, it is known from \cite{auer2003nonstochastic} that even in the iid bandit problem, for any time horizon $T$ there exists a set of reward distributions such that no algorithm can achieve a regret bound better than $O(\sqrt{T})$, while logarithmic regret upper bounds dependent on instance-specific constants have been proved by many existing works (see, e.g., \cite{auerbianchi2002, anantharamvaraiya1987-1}) in the iid setting.  
}

The remainder of this paper is organized as follows. Related work is given in Section \ref{single:sec:strong-related}. In
Section \ref{single:sec:strong-probform}, we present the problem formulation, notations and some preliminaries, 
including an equivalent countable representation of the information state. 
Then in Section \ref{single:sec:strong-alg} we present a learning algorithm, followed by a finite partition of the information states in Section \ref{single:sec:defnfinpar}, which is used to bound the strong regret of the learning algorithm. 
We analyze the regret of this learning algorithm in Section \ref{single:sec:strong-analysis} and prove that it increases logarithmically in time.  
We then introduce two variants of the learning algorithm aimed at relaxing certain assumptions, in Sections \ref{sec:extensions_explore} and \ref{sec:extensions_finpar}, respectively. \newt{Extensive numerical results on the performance of our online learning algorithms and comparison with other online learning algorithms in prior works is given in Section \ref{sec:num}.}
Section \ref{sec:conclusion} concludes the paper. 

%% file: strong_related.tex
Related work on multi-armed bandit problems started with the seminal paper
by Lai and Robbins \cite{lairobbins1985}, where asymptotically optimal adaptive
policies for arms with iid reward processes (referred to as the ``iid problem'' for simplicity below) 
were constructed. 
\rdec{These are index policies, and
it was shown that they achieve an asymptotically optimal $O(\log t)$ regret, \newt{for single parameterized bandit problems}, meaning that the regret is optimal both in terms of the logarithmic order in time $t$ (referred to as order optimality), and optimal among all algorithms with logarithmic order in time (referred to as optimality in the constant), asymptotically. } 
Later, Agrawal \cite{agrawal1995} considered the same iid problem and provided sample mean based index policies which are easier to compute, order optimal but not 
optimal in terms of the constant in general. {Multiple simultaneous plays was considered in Anantharam et al \cite{anantharamvaraiya1987-1}, and an asymptotically optimal policy was proposed.}
All the above work assumed parametrized distributions for the reward process of the arms. \newt{The problem class we consider in this paper, i.e., URBP, includes parameterized iid bandits with finite  number of states such as the Bernoulli bandits. Therefore, there exists instances of URBP for which no online learning algorithm can achieve better than logarithmic regret. However, in general proving lower bounds for bandit problems is a challenging task and is out of the scope of this paper. Rather than proving lower bounds on the strong regret, in this paper we focus on proving upper bounds on the strong regret, and evaluating the real time performance of our learning algorithms numerically. 
}
%
\newt{In contrast to the work on parameterized bandits}, Auer et al \cite{auerbianchi2002} proposed sample mean based index policies for the iid problem with logarithmic regret when reward processes have a bounded support.
Their upper bound holds uniformly over time rather than asymptotically but this bound is not asymptotically optimal.   

There have been efforts to extend this type of learning from iid bandits to Markovian bandits. Markovian bandits are further divided into two groups: {\em rested} bandits whereby the state of an arm can only change when it is played or activated, and {\em restless} bandits whereby the state of an arm changes according to different Markovian rules depending on whether it is played or not. Optimization version of the rested bandits, when the player knows the transition probabilities, for a discounted reward criterion was solved by Gittins and Jones \cite{gittins1972dynamic}, while the restless bandits was shown to be intractable by Papadimitriou and Tsitsiklis \cite{papadim1999} even in the optimization version. Nevertheless, heuristics, approximations and exact solutions under different/stronger assumptions on the arms have been studied by many, see e.g., \cite{whittle1988restless, guha2010approximation, ahmad2009optimality}. {The learning version of the rested bandits was considered by Anantharam et. al. \cite{anantharamvaraiya1987-2}, and an asymptotically optimal algorithm was proposed.}
%
Note that when the arms are driven by iid processes, the optimal policy (given all statistics) is a static policy that always plays the arm with the highest expected reward. Thus in this case weak regret coincides with strong regret, and the problem is greatly simplified in that arm selection is based on a function of the sample mean of the sequence of rewards from the arms, but not on the sequence itself.  The same observation holds when the arms evolve according to a rested Markovian process, since the optimal policy \rdec{for average reward} is also a static policy when the time horizon is sufficiently large.

Because of the aforementioned difficulties in efficiently computing the optimal dynamic policy \rdec{for the restless bandits} even when the transition probabilities are known, the performance criteria used in designing computationally efficient learning algorithms for restless bandits is typically the weak regret, by comparing to the best static, or single-action policy, one that always plays the arm with the highest expected reward. 
%
In particular, following the approach by Auer et al \cite{auerbianchi2002}, we in \cite{tekin2010online, tekin2011online} provided policies with uniformly
logarithmic {weak} regret bounds with respect to the best-single arm 
policy for both restless and rested bandit problems and
extended the results to single-player multiple-play and decentralized multi-player models in \cite{tekin2012online}. We achieved this by merging blocks of observed states from the same arm to form a continuous sample path of the underlying Markov process. In a parallel work by Liu et. al. \cite{liuliu2010}, \rdec{a similar result is obtained through deterministic sequencing of exploration and exploitation, in which the player explores or exploits an arm in blocks whose lengths increase geometrically over time.}  
Also related are decentralized multi-player versions of the iid problem under different collision models, see e.g., \cite{tekin2011performance, tekin2012sequencing, liu2010distributed, anandkumar2011distributed}.

There has been relatively less work in designing learning algorithms with stronger regret notions than the weak regret, with the exception of \cite{daigai2011}, in which a bandit problem with identical two-state arms was considered, whereby the optimal policy belongs to a finite set given to the player, 
and \cite{tekin2011approx}, in which we considered a special case of the restless bandits called the {\em feedback} bandits, a more general version of the problem studied in \cite{daigai2011} and proposed a computationally efficient learning algorithm that is approximately optimal compared to the optimal dynamic policy, based on an optimization algorithm proposed by Guha et. al. \cite{guha2010approximation}. 

Outside the bandit literature, there has been a lot of work in adaptive learning of Markov Decision Processes (MDPs) with finite state and action spaces, {where the goal is to learn the optimal policy with the smallest possible strong regret.} 
Burnetas and Katehakis \cite{burnetas1997optimal}
proposed index policies with asymptotically logarithmic regret, where 
the indices are the inflations of right-hand-sides of the estimated average reward optimality equations based on Kullback Leibler (KL) divergence, and showed that these are asymptotically optimal both in terms of the order and the constant. However, they assumed that
the support of the transition probabilities is known. Tewari and
Bartlett \cite{tewari2008optimistic} proposed a learning algorithm that uses
$l_1$ distance instead of KL divergence with the same order of
regret but a larger constant; their proof is simpler than that found in \cite{burnetas1997optimal} and does not require the support of the transition probabilities to be known. Auer and Ortner \cite{auer2009near} proposed
another algorithm with logarithmic regret and reduced computation
for the MDP problem, which solves the average reward optimality
equations only when a confidence interval is halved. In all the above work the MDPs are assumed to be irreducible.

{The rested bandit problem with Markovian reward processes may be viewed as a special case of the MDP problem, and therefore the above literature applies to the rested bandit learning and can be used to obtain strong regret results.  By contrast, the restless bandit problem is a special case of the partially observed MDP (or POMDP) problems, which has an uncountably infinite state space and the irreducibility condition may not hold in general.  The goal of this paper is to develop learning algorithms that achieve logarithmic strong regret for the restless bandit problem.} 

{It is also worth noting the difference between regret learning, the approach taken in this paper, and Q-learning. The learning approach we take in this paper is model-based, in which the player estimates the transition probabilities and exploits the estimates to learn how to play. By contrast, Q-learning is a model-free approach that estimates Q-values for state-action pairs (rather than transition probabilities). However, the convergence guarantees of Q-learning algorithms are weaker than our regret bounds. In general, Q-learning algorithms do not have sublinear regret guarantees, and convergence in terms of average reward only takes place when all state-action pairs are observed infinitely many times. See \cite{watkins1992q, kaelbling1996reinforcement, du1995q} for examples of Q-learning in finite MDPs and \cite{kimura1997reinforcement, james2009sarsalandmark} for Q-learning in POMDPs.
Specifically in \cite{du1995q}, a Q-learning algorithm for the rested bandit problem is proposed , in which estimates of the Gittins indices for the arms are obtained via Q-learning. In \cite{james2009sarsalandmark}, a POMDP problem with {\em landmark states} in which each state results in a unique observation is studied, and under standard stochastic approximation conditions asymptotic convergence of Q-functions to the optimal Q-values was proved.
}


%% file: strong_probform.tex
Consider $K$ mutually independent uncontrolled restless Markovian arms, indexed by the set ${\cal K} = \{1,2,\ldots,K\}$ whose states evolve in discrete time steps $t=1,2,\ldots$ according to a finite-state Markov chain with unknown transition probabilities. 

Let $S^k$ be the state space of arm $k$. For simplicity of presentation, we assume that for state $x \in S^k$, $r^k_{x} = x$, i.e., the state of an arm also represents its reward under that state.  {This is without loss of generality as long as one of the following is true: either the state is perfectly observed when played, or that the reward is perfectly observed when received which uniquely identifies a state for a given arm (i.e., no two states have the same reward).}  
It follows that the state space of the system is the Cartesian product of the state spaces of individual arms, denoted by $\boldsymbol{S}= S^1 \times \ldots \times S^K$.  
Let $p^k_{ij}$ denote the transition probability from state $i$ to state $j$ of arm $k$. The transition probability matrix of arm $k$ is denoted by $P^k$, whose $(i,j)$th element is $p^k_{ij}$. The set of transition probability matrices is denoted by $\boldsymbol{P} = (P^1, \ldots, P^K)$. 
\revq{We assume that $P^k$s are such that each arm is ergodic.} \revq{This implies that}, for each arm there exists a unique stationary distribution which is given by $\boldsymbol{\pi}^k =  (\pi^k_x)_{x \in S^k}$.
At each time step, the state of the system is a $K$-dimensional vector of states of arms which is given by $\boldsymbol{x}=(x^1, \ldots, x^K) \in \boldsymbol{S}$. 

\revq{The following notation will be frequently used throughout the paper.} Let $e^k_x$ represent the unit vector with dimension $|S^k|$, whose $x$th element is $1$, and all other elements are $0$.  
$\mathbb{N}=\{1,2,\ldots\}$ denotes the set of natural numbers, $\mathbb{Z}_+=\{0,1,\ldots\}$ the set of non-negative integers, $( \boldsymbol{v} \bullet \boldsymbol{w})$ the standard inner product of vectors $\boldsymbol{v}$ and $\boldsymbol{w}$, $||\boldsymbol{v}||_1$ and $||\boldsymbol{v}||_\infty$ respectively the $l_1$ and $l_\infty$ norms of vector $\boldsymbol{v}$, and $||P||_1$ the induced maximum row sum norm of matrix $P$. For a vector $\boldsymbol{v}$, $(\boldsymbol{v}_{-u},v')$ denotes the vector whose $u$th element is $v'$, while all other elements are the same as in $\boldsymbol{v}$. For a vector of matrices $\boldsymbol{P}$, $(\boldsymbol{P}_{-u},P')$ denotes the vector of matrices whose $u$th matrix is $P'$, while all other matrices are the same as in $\boldsymbol{P}$. The transpose of a vector $\boldsymbol{v}$ or matrix $P$ is denoted by $\boldsymbol{v}^T$ or $P^T$, respectively. 
In addition, the following quantities frequently appear in this paper: 
%
\begin{itemize}
\item 
$\beta= \sum_{t=1}^{\infty} 1/t^{2}$, $\pi^k_{\min} = \min_{x \in S^k} \pi^k_x$; 
\item 
$\pi_{\min} = \min_{k \in \mathcal{K}} \pi^k_{\min}$; 
\item 
$r_{\max}=\max_{x \in S^k, k \in \mathcal{K}} r^k_x$; 
\item 
$S_{\max}=\max_{k \in \mathcal{K}} |S^k|$. 
\end{itemize}

There is a player who selects one of the $K$ arms at each time step $t$, and gets a {bounded} reward depending on the state of the selected arm at time $t$. {Without loss of generality, we assume that the state rewards are non-negative.} Let {$r^k(t)$ be the random variable} which denotes the reward from arm $k$ at time $t$.
The objective of the player is to maximize the undiscounted sum of the rewards over any finite horizon $T>0$. However, the player does not know the set of transition probability matrices $\boldsymbol{P}$. In addition, at any time step $t$ the player can only observe the state of the arm it selects but not the states of the other arms. Intuitively, in order to maximize its reward, the player needs to explore/sample the arms to estimate their transition probabilities and to reduce the uncertainty about the current state $\boldsymbol{x} \in \boldsymbol{S}$ of the system, while it also needs to exploit the information it has acquired about the system to select arms that yield high rewards.  The exploration and exploitation need to be carefully balanced to yield the maximum reward for the player. In a more general sense, the player is learning to play optimally in an uncontrolled POMDP. 

We denote the set of {all possible stochastic matrices with $|S^k|$ rows and $|S^k|$ columns} by $\Xi^k$, and let 
$\boldsymbol{\Xi} = (\Xi^1, \Xi^2, \ldots, \Xi^K)$.  
Since $\boldsymbol{P}$ is unknown to the player, at time $t$ the player has an estimate of $\boldsymbol{P}$, denoted by $\hat{\boldsymbol{P}}_t \in \boldsymbol{\Xi}$. For two {vectors} of transition probability matrices $\boldsymbol{P}$ and $\tilde{\boldsymbol{P}}$, the distance between them is defined as $||\boldsymbol{P}- \tilde{\boldsymbol{P}}||_1 := \sum_{k=1}^K ||P^k - \tilde{P}^k ||_1$.
%
Let $X^k_t$ be the random variable representing the state of arm $k$ at time $t$. Then, the random vector $\boldsymbol{X}_t = (X^1_t, X^2_t, \ldots, X^K_t)$ represents the state of the system at time $t$. 

{The action space $U$ of the player is equal to ${\cal K}$ since it chooses an arm in ${\cal K}$ at each time step, and the observation space $Y$ of the player is equal to $\cup_{k=1}^K S^k$, since it observes the state of the arm it selects at each time step.}
{Since the player can distinguish different arms, for simplicity we will assume $S^k \cap S^l = \emptyset$ for $k \neq l$, so that these states may be labeled distinctly.} {Let $u_t \in U$ be the arm selected by the player at time $t$, and $y_t \in Y$ be the state/reward observed by the player at time $t$. The history of the player at time $t$ consists of all the actions and observations of the player by time $t$, which is denoted by $\boldsymbol{z}^t = (u_1, y_1, u_2, y_2, \ldots, u_{t}, y_t)$. Let $H^t$ denote the set of histories at time $t$.} 
{An algorithm $\alpha = (\alpha(1), \alpha(2), \ldots)$ for the player, is a sequence of mappings from the set of histories to actions, i.e., $\alpha(t) : H^t \rightarrow U$. Since the history depends on the stochastic evolution of the arms, let $U_t$ and $Y_t$ be the random variables representing the action and the observation at time $t$, respectively.  }
Let $Q_{\boldsymbol{P}}(y|u)$ be the sub-stochastic transition probability matrix such that
\begin{align*}
(Q_{\boldsymbol{P}}(y|u))_{\boldsymbol{x} \boldsymbol{x}'} = P_{\boldsymbol{P}}(\boldsymbol{X}_{t} =\boldsymbol{x}', Y_{t}=y | \boldsymbol{X}_{t-1} = x, U_{t} = u),
\end{align*}
where $P_{\boldsymbol{P}}(.|.)$ denotes the conditional probability with respect to distribution $\boldsymbol{P}$. For URBP, $Q_{\boldsymbol{P}}(y|u)$ is the zero matrix for $y \notin S^u$, and for $y \in S^u$, only nonzero entries of $Q_{\boldsymbol{P}}(y|u)$ are the ones for which $x^u = y$.


Let $\Gamma$ be the set of admissible policies, i.e., policies $\gamma'$ for which $\gamma'(t) : H^t \rightarrow U$. Note that the set of admissible policies include the set of optimal policies which are computed by dynamic programming based on $\boldsymbol{P}$. Let $\psi_0$ be the initial belief of the player, which is a probability distribution over $\boldsymbol{S}$. Since we assume that the player knows nothing about the state of the system initially, $\psi_0$ can be taken as the uniform distribution over $\boldsymbol{S}$.

Let $E^{\boldsymbol{P}}_{\psi,\gamma}[.]$ denote the expectation taken with respect to {an algorithm or policy $\gamma$}, initial state $\psi$, and the set of transition probability matrices $\boldsymbol{P}$.
The performance of an algorithm $\alpha$ can be measured by its strong regret, whose value at time $t$ is the difference between performance of the algorithm and performance of the optimal policy by time $t$. It is given by
\begin{align}
R^{\alpha}(T) = \sup_{\gamma' \in \Gamma} \left( E^{\boldsymbol{P}}_{\psi_0, \gamma'}\left[\sum_{t=1}^T  r^{\gamma'(t)}(t)\right]\right) - E^{\boldsymbol{P}}_{\psi_0, \alpha}\left[\sum_{t=1}^T r^{\alpha(t)}(t)\right]. \label{eqn:probform1}
\end{align}

{It is easy to see that the time average reward of any algorithm with sublinear regret, i.e., regret $O(T^{\rho})$, $\rho <1$, converges to the time average reward of the optimal policy.  For any algorithm with sublinear regret, its regret is a measure of its convergence rate to the average reward. In Section \ref{single:sec:strong-alg}, we will give an algorithm whose regret grows logarithmically in time, which is the best possible rate of convergence.}


%% file: strong_aroe.tex
As mentioned earlier, if the transition probability matrices of the arms are known by the player, then the URBP becomes an optimization problem (POMDP) rather than a learning problem.  In this section we discuss the solution approach to this optimization problem.  This approach is then used in subsequent sections by the player in the learning context using {\em estimated} transition probability matrices.  

%
A POMDP problem is often presented using the belief space (or information state), i.e., the set of probability distributions over the state space. For the URBP with the set of transition probability matrices $\boldsymbol{P}$, the belief space is given by 
\begin{align*}
\boldsymbol{\Psi} := \left\{\psi : \psi^T \in \mathbb{R}^{|\boldsymbol{S}|}, \psi_{\boldsymbol{x}} \geq 0, \forall \boldsymbol{x} \in \boldsymbol{S}, \sum_{\boldsymbol{x} \in \boldsymbol{S}} \psi_{\boldsymbol{x}} =1 \right\}, 
\end{align*}
which is the unit simplex in $\mathbb{R}^{|\boldsymbol{S}|}$. Let $\psi_t$ denote the belief of the player at time $t$.
Then the probability that the player observes $y$ given it selects arm $u$ when the belief is $\psi$ is given by 
\begin{align*}
V_{\boldsymbol{P}}(\psi,y,u) := \psi Q_{\boldsymbol{P}}(y|u) \boldsymbol{1},
\end{align*}
where $\boldsymbol{1}$ is the $|\boldsymbol{S}|$ dimensional column vector of $1$s. Given arm $u$ is chosen under belief state $\psi$ and $y$ is observed, the next belief state is 
\begin{align*}
T_{\boldsymbol{P}}(\psi,y,u) := \frac{\psi Q_{\boldsymbol{P}}(y|u)}{V_{\boldsymbol{P}}(\psi,y,u)} ~.
\end{align*}
%
The average reward optimality equation (AROE) is
\begin{align}
g + h(\psi) &= \max_{u \in U} \left\{ \bar{r}(\psi,u) + \sum_{y \in S^u} V_{\boldsymbol{P}}(\psi,y,u) h(T_{\boldsymbol{P}}(\psi,y,u)) \right\}, \label{s:urestls:eqn:AROE}
\end{align}
where $g$ is a constant and $h$ is a function from $\boldsymbol{\Psi} \rightarrow \mathbb{R}$,
\begin{align*}
\bar{r}(\psi,u)=(\psi \bullet r(u))= \sum_{x^u \in S^u} x^u \phi_{u,x^u}(\psi)
\end{align*}
is the expected reward of action $u$ under belief $\psi$, $\phi_{u,x^u}(\psi)$ is the probability that arm $u$ is in state $x^u$ given belief $\psi$, $r(u)=(r(\boldsymbol{x},u))_{\boldsymbol{x} \in S}$ and $r(\boldsymbol{x},u)=x^u$ is the reward when arm $u$ is chosen in state $\boldsymbol{x}$. 

%
\begin{assumption} \label{s:urestls:assump:1}
$p^k_{ij}>0, \forall k \in {\cal K}, i,j \in S^k$.
\end{assumption}
{When Assumption \ref{s:urestls:assump:1} holds}, the existence of a bounded, convex continuous solution to (\ref{s:urestls:eqn:AROE}) is guaranteed. 
\newt{The set of Markov chains for which Assumption \ref{s:urestls:assump:1} holds is a subset of the class of aperiodic Markov chains. From any periodic or aperiodic Markov chain, we can obtain a Markov chain which belongs to this class by adding a small uniform perturbation to the state transition probabilities.}

Let $V$ denote the space of bounded real-valued functions on $\boldsymbol{\Psi}$. Next, we define the undiscounted dynamic programming operator $F: V \rightarrow V$. Let $v \in V$, we have
\begin{align}
(Fv)(\psi) = \max_{u \in U} \left\{ \bar{r}(\psi,u) + \sum_{y \in S^u} V_{\boldsymbol{P}}(\psi,y,u) v(T_{\boldsymbol{P}}(\psi,y,u)) \right\}. \label{s:urestls:eqn:DPOP}
\end{align}
{In the following lemma, we give some of the properties of the solutions to the average reward optimality equation and the dynamic programming operator defined above.} 
\begin{lemma} \label{s:urestls:lemma:ACOEexistence}
Let $h_+ = h-\inf_{\psi \in \boldsymbol{\Psi}} (h(\psi))$, $h_- = h-\sup_{\psi \in \boldsymbol{\Psi}} (h(\psi))$ and 
\begin{align*}
h_{T,\boldsymbol{P}}(\psi) = \sup_{\gamma \in \Gamma} \left( E^{\boldsymbol{P}}_{\psi, \gamma}\left[\sum_{t=1}^T r^\gamma(t)\right]\right).
\end{align*}
{Given that Assumption \ref{s:urestls:assump:1} is true,} the following holds:
\begin{enumerate}[{S}-1]
\item  Consider a sequence of functions $v_0, v_1, v_2, \ldots$ in $V$ such that $v_0=0$, and $v_l = F v_{l-1}$, $l=1,2,\ldots$. This sequence converges uniformly to a convex continuous function $v^*$ for which $Fv^* = v^* + g$ where $g$ is a finite constant. In terms of (\ref{s:urestls:eqn:AROE}), this result means that there exists a finite constant $g_{\boldsymbol{P}}$ and a bounded convex continuous function $h_{\boldsymbol{P}}: \boldsymbol{\Psi} \rightarrow \mathbb{R}$ which is a solution to (\ref{s:urestls:eqn:AROE}).
\item $h_{\boldsymbol{P}-}(\psi) \leq h_{T,\boldsymbol{P}}(\psi)-T g_{\boldsymbol{P}} \leq h_{\boldsymbol{P}+}(\psi)$, $\forall \psi \in \boldsymbol{\Psi}$.
\item $h_{T,\boldsymbol{P}}(\psi) = T g_{\boldsymbol{P}} + h_{\boldsymbol{P}}(\psi) + O(1)$ as $T \rightarrow \infty$.
\end{enumerate}
\end{lemma}
\begin{proof}
Sufficient conditions for the existence of a bounded convex continuous solution to the AROE are investigated in \cite{platzman1980}. According to Theorem 4 of \cite{platzman1980}, if reachability and detectability conditions are satisfied then S-1 holds. Below, we directly prove that reachability condition in \cite{platzman1980} is satisfied. To prove that detectability condition is satisfied, we show another condition, i.e., subrectangular substochastic matrices, holds which implies the detectability condition.
 
We note that $P(\boldsymbol{X}_{t+1} = \boldsymbol{x}' | \boldsymbol{X}_{t} = \boldsymbol{x}  ) > 0$, $\forall \boldsymbol{x}, \boldsymbol{x}' \in \boldsymbol{S}$ since by Assumption \ref{s:urestls:assump:1}, $p^k_{ij} > 0$ $\forall i,j \in S^k, \forall k \in {\cal K}$.

\begin{condition}\label{cond:platzman2}
(Reachability) There is a $\rho < 1$ and an integer $\xi$ such that {for all $\boldsymbol{x} \in \boldsymbol{S}$}
\begin{align*}
\sup_{\gamma \in \Gamma} \max_{0 \leq t \leq \xi} P(\boldsymbol{X}_t = \boldsymbol{x} | \psi_0) \geq 1 - \rho, ~~ \forall \psi_0 \in \boldsymbol{\Psi}.
\end{align*}
\end{condition}
Set $\rho = 1 -\min_{\boldsymbol{x}, \boldsymbol{x}'} P(\boldsymbol{X}_{t+1} = \boldsymbol{x}' | \boldsymbol{X}_{t} = \boldsymbol{x}  )$, $\xi = 1$. Since the system is uncontrolled, state transitions are independent of the arm selected by the player. Therefore,
\begin{align*}
\sup_{\gamma \in \Gamma}  P(\boldsymbol{X}_1 = \boldsymbol{x} | \psi_0) &= P(\boldsymbol{X}_1 = \boldsymbol{x} | \psi_0) \\
& \geq \min_{\boldsymbol{x}, \boldsymbol{x}'} P(\boldsymbol{X}_{t+1} = \boldsymbol{x}' | \boldsymbol{X}_{t} = \boldsymbol{x}  ) = 1 - \rho.
\end{align*}
\begin{condition}\label{cond:platzman3}
(Subrectangular matrices) For any substochastic matrix $Q(y|u), y \in Y, u \in U$, and for any $i,i',j,j' \in \boldsymbol{S}$,
\begin{align*}
(Q(y|u))_{ij} >0 \textrm{ and } (Q(y|u))_{i'j'} > 0 ~~\Rightarrow (Q(y|u))_{ij'} >0 \textrm{ and } (Q(y|u))_{i'j} > 0.
\end{align*}
\end{condition}
$Q(y|u)$ is subrectangular for $y \notin S^u$ since it is the zero matrix. For $y \in S^u$ all entries of $Q(y|u)$ is positive since $P(\boldsymbol{X}_{t+1} = \boldsymbol{x}' | \boldsymbol{X}_{t} = \boldsymbol{x}  ) > 0$, $\forall \boldsymbol{x}, \boldsymbol{x}' \in \boldsymbol{S}$.

S-2 holds by Lemma 1 in \cite{platzman1980}, and S-3 is a consequence of S-2 and the boundedness property in S-1. 
\end{proof}

%% file: strong_countable.tex
The belief space is uncountable. Since the problem we consider is a learning problem, it is natural to assume that the player does not have an initial belief about the state of the system.  However, in a learning context there is no loss of generality in adopting an initial belief formed by playing each arm at least once.  Assume that the initial $K$ steps are such that the player selects arm $k$ at the $k$th step. Then the POMDP for the player can be written as a countable-state MDP.  In this case a more succinct way of representing the information state at time $t$ is given by  
\begin{align*}
(\boldsymbol{s}_t, \boldsymbol{\tau}_t) = ( (s^1_t, s^2_t \ldots, s^K_t), (\tau^1_t, \tau^2_t \ldots, \tau^K_t)),
\end{align*}
where $s^k_t$ and $\tau^k_t$ are the last observed state of arm $k$ and how long ago (from $t$) the last observation of arm $k$ was made, respectively.
Note that the countable state MDP obtained this way is a subset of the POMDP for the bandit problem in which the player can only be in one of the countably many points in the belief space $\boldsymbol{\Psi}$ at any time step $t$. Our approach is to exploit the continuity property of the AROE to bound the regret of the player. In order to do this we need to work with both methods of state representation. 
We thus make a distinction between $\psi_t$, which is a probability distribution over the state space, and $(\boldsymbol{s}_t, \boldsymbol{\tau}_t)$ which is a sufficient statistic for the player to calculate $\psi_t$ when $\boldsymbol{P}$ is given.  Subsequently we will call $\psi \in \boldsymbol{\Psi}$, the {\em belief} or {\em belief state}, and $(\boldsymbol{s}, \boldsymbol{\tau})$ the {\em information state}.\footnote
{We note that in the POMDP literature these two terms are generally used interchangeably.}

{The contribution of the initial $K$ steps to the regret is at most $K r_{\max}$, which we will subsequently ignore in our analysis}. We will only analyze the time steps after this initialization, and set $t=0$ upon the completion of the initialization phase. The initial information state of the player can be written as  $(\boldsymbol{s}_0, \boldsymbol{\tau}_0)$. Let ${\cal C}$ be the set of all possible information states that the player can be in. Since the player selects a single arm at each time step, at any time$t$, $\tau^k_t=1$ for the last selected arm $k$ (at $t-1$).

The player can compute its belief state $\psi_t \in \Psi$ by using its transition probability estimates $\hat{\boldsymbol{P}}_t$ together with the information state $(\boldsymbol{s}_t, \boldsymbol{\tau}_t)$. We let $\psi_{\boldsymbol{P}}(\boldsymbol{s}_t, \boldsymbol{\tau}_t)$ be the belief that corresponds to information state $(\boldsymbol{s}_t, \boldsymbol{\tau}_t)$ when the set of transition probability matrices is $\boldsymbol{P}$. The player knows the information state exactly, but it only has an estimate of the belief that corresponds to the information state, because it does not know the transition probabilities. The true belief computed with the knowledge of exact transition probabilities and information state at time $t$ is denoted by $\psi_t$, while the estimated belief computed with estimated transition probabilities and information state at time $t$ is denoted by $\hat{\psi}_t$.

When the belief is $\psi$ and the set of transition probability matrices is $\boldsymbol{P}$, the set of optimal actions which are the maximizers of (\ref{s:urestls:eqn:AROE}) is denoted by $O(\psi;\boldsymbol{P})$. When the information state is $(\boldsymbol{s}_t, \boldsymbol{\tau}_t)$, and the set of transition probability matrices is $\boldsymbol{P}$, \revq{we denote the set of optimal actions by $O((\boldsymbol{s}, \boldsymbol{\tau}); \boldsymbol{P}) := O(\psi_{\boldsymbol{P}}((\boldsymbol{s}, \boldsymbol{\tau}));\boldsymbol{P})$.}


%% file: strong_alg.tex
\begin{figure}[h!]
\hspace{-0.15in}
\fbox {
\begin{minipage}{\columnwidth}
\flushleft{Average Reward with Estimated Probabilities (AREP)}
{\fontsize{10}{10}\selectfont
\begin{algorithmic}[1]
\STATE{Initialize: $f(t)$ given for $t \in \{1,2,\ldots \}$, $t=1$, $N^k_{i,j}=0, C^k_i=0$, $\forall k \in {\cal K}, i,j \in S^k$. Play each arm once to set the initial information state $(\boldsymbol{s}_0, \boldsymbol{\tau}_0)$. Pick $\alpha(0)$ randomly.}
\WHILE {$t\geq 1$}
\STATE{$\hat{p}^k_{ij} = ( I(N^k_{i,j}=0) + N^k_{i,j})/(|S^k| I(C^k_i=0) +C^k_i)$}
%
\STATE{$W=\{ k \in {\cal K}: \textrm{ there exists } i \in S^k \textrm{ such that } C^k_i < f(t) \}$.}
\IF{$W \neq \emptyset$}
\STATE{EXPLORE}
\IF{$\alpha(t-1) \in W$}
\STATE{$\alpha(t)=\alpha(t-1)$}
\ELSE
\STATE{select $\alpha(t) \in W$ arbitrarily}
\ENDIF
\ELSE
\STATE{EXPLOIT}
\STATE{solve $\hat{g_t} + \hat{h_t}(\psi) = \max_{u \in U} \{ \bar{r}(\psi,u) + \sum_{y \in S^u} V(\psi,y,u) \hat{h}_t(T_{\hat{\boldsymbol{P}}_t}(\psi,y,u)) \}, \forall \psi \in \Psi$.}
\STATE{Let $\hat{\psi_t}$ be the estimate of the belief at time $t$ based on $(\boldsymbol{s}_t, \boldsymbol{\tau}_t)$ and $\hat{\boldsymbol{P}}_t$.}
\STATE{compute the indices of all actions at $\hat{\psi}_t$:}
\STATE{$\forall u \in U$, ${\cal I}_t (\hat{\psi}_t, u)= \bar{r}(\hat{\psi}_t,u) + \sum_{y \in S^u} V(\hat{\psi}_t,y,u) \hat{h}_t(T_{\hat{\boldsymbol{P}}_t}(\hat{\psi}_t,y,u))$.}
\STATE{Let $u^*$ be the arm with the highest index (arbitrarily select one if there is more than one such arm).}
\STATE{$\alpha(t)=u^*$.}
\ENDIF
\STATE{Receive reward $r^{\alpha(t)}(t)$, i.e., state of $\alpha(t)$ at $t$}
\STATE{Compute $(\boldsymbol{s}_{t+1}, \boldsymbol{\tau}_{t+1})$}
\IF{$\alpha(t-1) = \alpha(t)$}
\FOR{$i,j \in S^{\alpha(t)}$}
\IF{State $j$ is observed at $t$, state $i$ is observed at $t-1$}
\STATE{$N^{\alpha(t)}_{i,j}=N^{\alpha(t)}_{i,j}+1$, $C^{\alpha(t)}_i=C^{\alpha(t)}_i+1$.}
\ENDIF
\ENDFOR
\ENDIF
\STATE{$t:=t+1$}
\ENDWHILE
\end{algorithmic}
}
\end{minipage}
} \caption{Pseudocode for the Average Reward with Estimated Probabilities (AREP) algorithm.} \label{s:urestls:fig:adaptive}
\end{figure}

In this section we propose the algorithm {\em Average Reward with Estimated Probabilities} (AREP) given in Fig. \ref{s:urestls:fig:adaptive}, as a learning algorithm for the player. AREP consists of exploration and exploitation phases. In the exploration phase the player plays each arm for a certain amount of time to form estimates of the transition probabilities, while in the exploitation phase the player selects an arm according to the optimal policy based on the estimated transition probabilities. 
At each time step, the player decides if it is an exploration phase or an exploitation phase based on the accuracy of the transition probability estimates. Let $N^k(t)$ be the number of times arm $k$ is selected by time $t$, $N^k_{i,j}(t)$ be the number of times a transition from state $i$ to state $j$ of arm $k$ is observed by the player by time $t$, and $C^k_i(t)$ be the number of times a transition from state $i$ of arm $k$ to any state of arm $k$ is observed by time $t$. Clearly,
\begin{align*}
C^k_i(t) = \sum_{j \in S^k} N^k_{i,j}(t).
\end{align*}
Let $f(t)$ be a non-negative, increasing function which sets a condition on the accuracy of the estimates. If $C^k_i(t) < f(t)$ for some $k \in {\cal K}$, $i \in S^k$, the player explores at time $t$. Otherwise, the player exploits at time $t$. 
In an exploration step, in order to update the estimate of $p^k_{ij}, j \in S^k$, the player keeps playing arm $k$ until state $i$ is observed, and then plays arm $k$ one more time to observe the  state following $i$.  \rdec{Note that in between the player can update other estimates depending on what states are observed.} 
Then the player forms the following sample mean estimates of the transition probabilities:
\begin{align*}
\hat{p}^k_{ij,t} := \frac{N^k_{i,j}(t)}{C^k_i(t)} ~, i,j \in S^k . 
\end{align*}
%
%
%

If AREP is in the exploitation phase at time $t$, then the player first computes $\hat{\psi_t}$, the estimated belief at time $t$, using the set of estimated transition probability matrices $\hat{\boldsymbol{P}}_t$. Then, it solves the AROE using $\hat{\boldsymbol{P}}_t$, to which the solution is given by $\hat{g}_t$ and $\hat{h}_t$.  For now we will ignore complexity issues and assume the player can compute the solution at every time step.  \rdec{More is discussed in the Conclusion.} 
This solution is used to compute the indices (given on line 17 of AREP) as
\begin{align*}
{\cal I}_t (\hat{\psi}_t, u)= \bar{r}(\hat{\psi}_t,u) + \sum_{y \in S^u} V(\hat{\psi}_t,y,u) \hat{h}_t(T_{\hat{\boldsymbol{P}}_t}(\hat{\psi}_t,y,u)),
\end{align*}
for each action $u \in U$ at estimated belief $\hat{\psi_t}$. ${\cal I}_t (\hat{\psi}_t, u)$ represents the advantage of choosing action $u$ starting from information state $\hat{\psi}_t$, i.e, the sum of gain and bias. After computing the indices for each action, the player selects the action with the highest index. In case of a tie, one of the actions with the highest index is randomly selected. Note that it is possible to update the state transition probabilities even in the exploitation phase given that the arms selected at times $t-1$ and $t$ are the same. Thus $C^k_i(t)$ may also increase in an exploitation phase, and the number of explorations may be smaller than the number of explorations needed in the worst case, where the estimates are only updated during exploration steps. 

In subsequent sections we will bound the strong regret of AREP by bounding the number of times a suboptimal arm selection is made at any information state. Since there are infinitely many information states, our approach to bounding the sum of the number of suboptimal plays is to introduce a finite partition of the space of information states. We do this in the next section. 
For the remainder of the paper we will denote AREP by $\alpha$.

%% file: strong_finpar.tex
Note that even if the player knows the optimal policy as a function of the belief state for any time horizon $T$, it may not be able to play optimally because it does not know the exact belief $\psi_t$ at time $t$.  One way to ensure that the player plays optimally in this case is to show that there exists an $\epsilon>0$ such that if $||\psi_t - \hat{\psi}_t||_1 < \epsilon$, the optimal actions in $\hat{\psi}_t$ belong to a subset of the set of optimal actions in $\psi_t$. This is indeed the case, and we prove it by exploiting the continuity of the solution to (\ref{s:urestls:eqn:AROE}) under Assumption \ref{s:urestls:assump:1}.


\subsection{Grouping the information states}

We start by introducing a finite partition of the set of information states ${\cal C}$. 
\begin{definition}\label{defn:partition}
Let $\tau_{\textrm{tr}} >0$ be an integer which denotes a threshold in time lag. This threshold is used to group all information states of an arm which has not been played for more than this threshold as a single group.  Consider a vector $\boldsymbol{i} = (i^1, \ldots, i^K)$ such that either $i^k=\tau_{\textrm{tr}}$ or $i^k=(s_i^k,\tau_i^k), \tau_i^k < \tau_{\textrm{tr}}, s_i^k \in S^k$. \rdec{For a finite $\tau_{\textrm{tr}}$ there are only a finite number of such vectors. Each vector defines a set of information states; it contains either a single  information state or infinitely many information states of the arms for which $i^k = \tau_{\textrm{tr}}$.  Together these vectors form a finite partition of ${\cal C}$, and we will call $\boldsymbol{i}$ a {\em partition vector}. 
%
Let ${\cal G}_{\tau_{\textrm{tr}}}$ denote the partition formed by $\tau_{\textrm{tr}}$, 
and let ${\cal M}(\boldsymbol{i}) := \{k: i^k =\tau_{\textrm{tr}}\}$ be the set of arms that are played at least $\tau_{\textrm{tr}}$ time steps ago, while $\overline{{\cal M}}(\boldsymbol{i}) := {\cal K}- {\cal M}(\boldsymbol{i})$. } 
Vector $\boldsymbol{i}$ represents the following set in the partition ${\cal G}_{\tau_{\textrm{tr}}}$: 
\begin{eqnarray}
%
\rdec{G_{\boldsymbol{i}} = \{(\boldsymbol{s}, \boldsymbol{\tau}) \in {\cal C}: s^k = s_i^k, \tau^k = \tau_i^k, \forall k\in \overline{{\cal M}}(\boldsymbol{i}), s^k \in S^k, \tau^k \geq \tau_{\textrm{tr}}, \forall k \in {\cal M}(\boldsymbol{i}) \}. 
} 
\end{eqnarray} 
\end{definition} 

Let $A(\tau_{\textrm{tr}})$ be the number of sets in partition ${\cal G}_{\tau_{\textrm{tr}}}$. Re-index the sets in ${\cal G}_{\tau_{\textrm{tr}}}$ as $G_1, G_2, \ldots, G_{A(\tau_{\textrm{tr}})}$. 
%
%
For a set $G_l \in {\cal G}_{\tau_{\textrm{tr}}}$, given a set of transition probability matrices $\boldsymbol{P}$, we define its center as follows. If $G_l$ only contains a single information state, then the belief corresponding to that information state is the center of $G_l$. If $G_l$ contains infinitely many information states, then the center belief of $G_l$ is the belief in which all arms for which $i^k = \tau_{\textrm{tr}}$ are in their stationary distribution based on $\boldsymbol{P}$. In both cases, the center belief of $G_l$ is denoted by $\psi^*(G_l; \boldsymbol{P})$. \newt{Let $\boldsymbol{s}_{l}$ be the information state corresponding to the center belief of $G_l$. Although $\psi^*(G_l; \boldsymbol{P})$ depends on $\boldsymbol{P}$, $\boldsymbol{s}_{l}$ does not depend on $\boldsymbol{P}$.}
\rdec{Since each arm is ergodic, when we map a set $G_l$ with infinitely many information states to the belief space using $\psi_{\boldsymbol{P}}$, for any $\delta>0$, only a finite number of information states in $G_l$ will lie outside the radius-$\delta$ ball around the center belief. } 

Let $O^*(G_l; \boldsymbol{P})$ be the set of optimal actions at the center belief. 
Note that as $\tau_{\textrm{tr}}$ increases, the number of sets with infinitely many elements increases, as does the 
number of sets with a single information state. 
The points in the belief space corresponding to these sets are shown in Figure \ref{fig:finite_partition}.  Below is an example of of such a finite partition of ${\cal C}$ with $\tau_{\textrm{tr}}=3$ when $K=2$.

\begin{example} \label{ex:afinitepart}
Let $K=2$, $S^1 = \{0,2\}$, $S^2 = \{1\}$ and $\tau_{\textrm{tr}}=3$. {For convenience we will rewrite $(\boldsymbol{s}, \boldsymbol{\tau}) = ((s^1, \tau^1), (s^2, \tau^2))$.} Then the partition formed by $\tau_{\textrm{tr}}$, i.e., ${\cal G}_{\tau_{\textrm{tr}}}$ contains the following sets:
\begin{align*}
G_1 &= \left\{ \left((0, 1), (1,2) \right)  \right\}, ~~~
G_2 = \left\{ \left((2, 1), (1,2) \right)  \right\}, \\
G_3 &= \left\{ \left((0, 2), (1,1) \right)  \right\}, ~~~
G_4 = \left\{ \left((2, 2), (1,1) \right)  \right\},\\
%
G_5 &= \left\{ \left((0, 1), (1,3) \right),~ \left((0, 1), (1,4)\right), \ldots \right\},\\
G_6 &=  \left\{ \left((2, 1), (1,3) \right),~ \left((2, 1), (1,4)\right), \ldots    \right\},\\
G_7 &= \left\{ \left((0, 3), (1,1) \right),~ \left((2, 3), (1,1) \right),~ \left((0, 4), (1,1)\right),~ \left((2, 4), (1,1) \right), \ldots    \right\}
\end{align*}
\end{example}

\begin{figure}
\begin{center}
\includegraphics[width=3.5in]{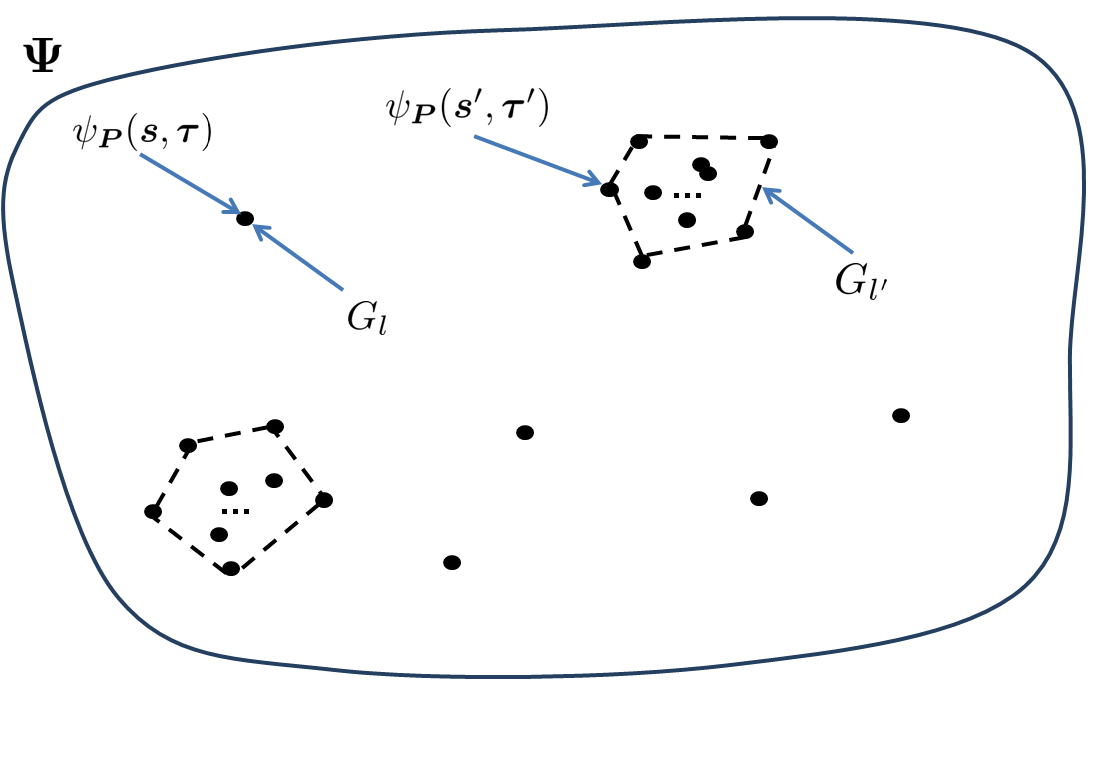}
\end{center}
\caption{Partition of ${\cal C}$ on $\boldsymbol{\Psi}$ based on $\boldsymbol{P}$ and $\tau_{\textrm{tr}}$. $G_l$ is a set with a single information state and $G_{l'}$ is a set with infinitely many information states.}
\label{fig:finite_partition}
\end{figure}

\subsection{Characterizing the set of optimal actions}

Next we define extensions of the sets $G_l$ on the belief space. For a set $B \in \boldsymbol{\Psi}$ let $B(\epsilon)$ be the {\em $\epsilon$-extension} of that set, i.e., 
\begin{align*}
B(\epsilon) = \{ \psi \in \boldsymbol{\Psi}: \psi \in B \textrm{ or } d_1(\psi,B) < \epsilon\},
\end{align*}
where $d_1(\psi,B)$ is the minimum $l_1$ distance between $\psi$ and any element of $B$. The $\epsilon$-extension of $G_l \in {\cal G}_{\tau_{\textrm{tr}}}$ corresponding to $\boldsymbol{P}$ is the $\epsilon$-extension of the convex-hull of the points $\psi_{\boldsymbol{P}}(\boldsymbol{s}, \boldsymbol{\tau})$ such that $(\boldsymbol{s}, \boldsymbol{\tau}) \in G_l$. Let $J_{l,\epsilon}$ denote the $\epsilon$-extension of $G_l$. Examples of $J_{l,\epsilon}$ on the belief space are given in Figure \ref{fig:epsilon_extension}.

\begin{figure}
\begin{center}
\includegraphics[width=3.5in]{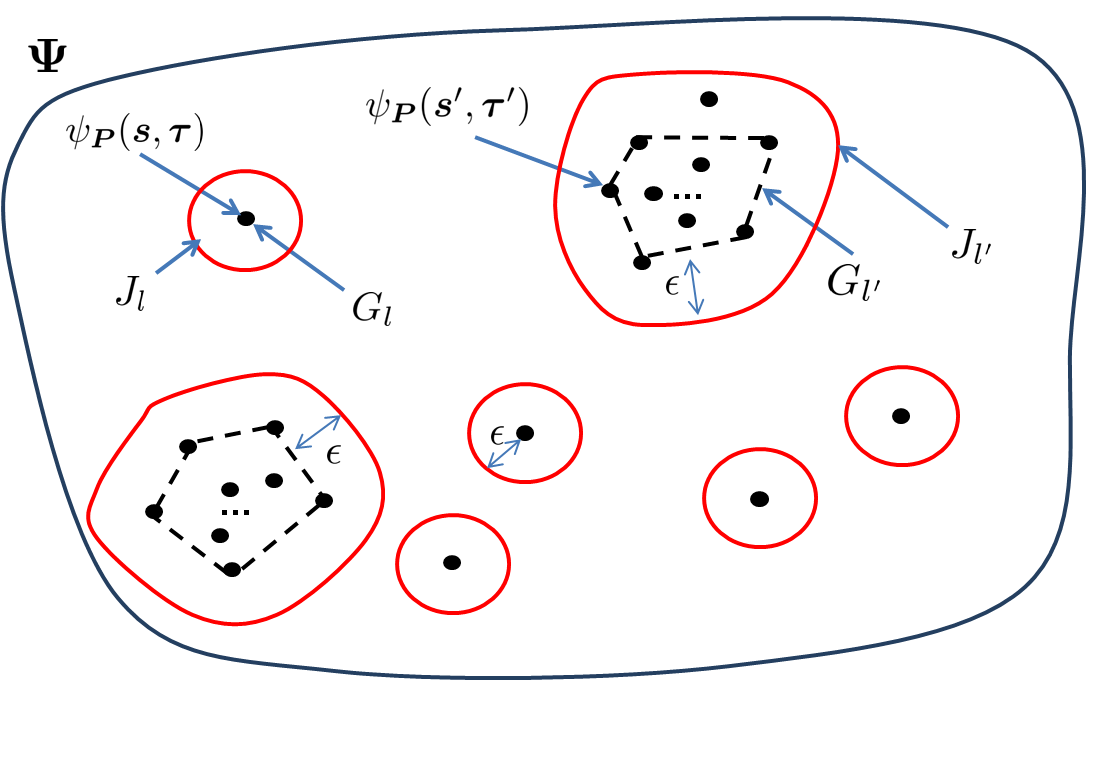}
\caption{$\epsilon$-extensions of the sets in ${\cal G}_{\tau_{\textrm{tr}}}$ on the belief space.}
\end{center}
\label{fig:epsilon_extension}
\end{figure}

Let the diameter of a set $B$ be the maximum distance between any two elements of \revq{that set}.  We note that when $\tau_{\textrm{tr}}$ increases, the diameter of the convex-hull of the points of an infinite-set in $ {\cal G}_{\tau_{\textrm{tr}}}$ decreases. 
The following lemma shows that when $\tau_{\textrm{tr}}$ is chosen large enough, there exists $\epsilon > 0$ such for all $G_l \in {\cal G}_{\tau_{\textrm{tr}}}$, we have non-overlapping $\epsilon$-extensions in which only a subset of the actions in $O^*(G_l; \boldsymbol{P})$ is optimal. 

\begin{lemma}\label{lemma:subsetoptimal}
For any $\boldsymbol{P}$ for which Assumption 1 holds, $\exists$ $\tau_{\textrm{tr}} > 0$ and $\epsilon>0$ such that for all $G_l \in {\cal G}_{\tau_{\textrm{tr}}}$, its $\epsilon$-extension $J_{l,\epsilon}$ has the following properties:
\begin{enumerate}[i]
\item For any $\psi \in J_{l,\epsilon}$, $O(\psi; \boldsymbol{P}) \subset O^*(G_l; \boldsymbol{P})$.
\item For $l \neq l'$, $J_{l,\epsilon} \cap J_{l',\epsilon} =\emptyset$.
\end{enumerate}
\end{lemma}
\begin{proof}
For $G_l \in {\cal G}_{\tau_{\textrm{tr}}}$ consider {its center} $\psi^*(G_l; \boldsymbol{P})$. For any $\psi \in \boldsymbol{\Psi}$ the suboptimality gap is defined as
\begin{align}
\Delta(\psi, \boldsymbol{P}) &= \max_{u \in U} \left\{ \bar{r}(\psi,u) + \sum_{y \in S^u} V_{\boldsymbol{P}}(\psi,y,u) h(T_{\boldsymbol{P}}(\psi,y,u))  \right\} \notag \\
&- \max_{u \in U - O(\psi; \boldsymbol{P})} \left\{ \bar{r}(\psi,u) + \sum_{y \in S^u} V_{\boldsymbol{P}}(\psi,y,u) h(T_{\boldsymbol{P}}(\psi,y,u)) \right\}. \label{eqn:beliefsuboptimality}
\end{align}
Since $r, h, V$ and $T$ are continuous in $\psi$, we can find an $\epsilon>0$ such that for any $\psi \in B_{2 \epsilon}(\psi^*(G_l; \boldsymbol{P}))$ and for all $u \in U$, 
\begin{align}
& \left| \bar{r}(\psi^*(G_l; \boldsymbol{P}),u) + \sum_{y \in S^u} V_{\boldsymbol{P}}(\psi^*(G_l; \boldsymbol{P}),y,u) h(T_{\boldsymbol{P}}(\psi^*(G_l; \boldsymbol{P}),y,u)) \right. \notag \\
& \left. - \bar{r}(\psi,u) + \sum_{y \in S^u} V_{\boldsymbol{P}}(\psi,y,u) h(T_{\boldsymbol{P}}(\psi,y,u)) \right| 
< \Delta(\psi^*(G_l; \boldsymbol{P}), \boldsymbol{P})/2,
\end{align}
and $B_{2 \epsilon}(\psi^*(G_l; \boldsymbol{P})) \cap B_{2 \epsilon}(\psi^*(G_{l'}; \boldsymbol{P})) = \emptyset$ for $l \neq l'$.
Therefore, any action $u$ which is not in $O^*(G_l; \boldsymbol{P})$ cannot be optimal for any $\psi \in B_{2 \epsilon}(\psi^*(G_l; \boldsymbol{P}))$.
Since the diameter of the convex-hull of the sets that contains infinitely many information states decreases with $\tau_{\textrm{tr}}$, there exists $\tau_{\textrm{tr}}>0$ such that for any $G_l \in {\cal G}_{\tau_{\textrm{tr}}}$, the diameter of the convex-hull $J_{l,0}$ is less than $\epsilon$. Let $\tau_{\textrm{tr}}$ be the smallest integer such that this holds. Then, the $\epsilon$-extension of the convex hull $J_{l,\epsilon}$ is included in the ball $B_{2 \epsilon}(\psi^*(G_l; \boldsymbol{P}))$ for all $G_l \in {\cal G}_{\tau_{\textrm{tr}}}$. This concludes the proof. 
\end{proof}

\begin{remark}\label{remark:unique}
According to Lemma \ref{lemma:subsetoptimal}, although we can find an $\epsilon$-extension in which a subset of $O^*(G_l; \boldsymbol{P})$ is optimal for any $\psi, \psi' \in J_{l,\epsilon}$, the set of optimal actions for $\psi$ may be different from the set of optimal actions for $\psi'$. Note that the player's estimated belief $\hat{\psi}_t$ is different from the true belief $\psi_t$. If no matter how close $\hat{\psi}_t$ is to $\psi_t$, their respective sets of optimal actions are different, then the player can make a suboptimal decision even if it knows the optimal policy.
Thus in this case the performance loss of the player, which can be bounded by the number of suboptimal decisions, may grow linearly over time. 
This turns out to be a major challenge.  In this paper, we present two different approaches that lead to performance loss (regret) growing logarithmically in time. The first approach is based on an assumption about the structure of the optimal policy, while the second approach is to construct an algorithm that will almost always choose near-optimal actions, whose sub-optimality can be controlled by a function of the time horizon $T$. 
\rdec{We shall take the first approach below and the second approach in Section \ref{sec:extensions_finpar}.} 
\end{remark}

\begin{assumption} \label{s:urestls:assump:esuboptimality}
There exists $\tau_{\textrm{tr}} \in \mathbb{N}$ such that for any $G_l \in {\cal G}_{\tau_{\textrm{tr}}}$, there exists $\epsilon>0$ such that the same subset of $O^*(G_l; \boldsymbol{P})$ is optimal for any $\psi \in J_{l,\epsilon} - \psi^*(G_l; \boldsymbol{P})$.
\end{assumption}

When this assumption is correct, if $\psi_t$ and $\hat{\psi}_t$ are sufficiently close to each other, then the player will always chose an optimal arm. Assume that this assumption is false. Consider the {\em stationary} information states for which $\tau^k = \infty$ for some arm $k$. Then for any $\tau_{\textrm{tr}} > 0$, there exists a set $G_l \in {\cal G}_{\tau_{\textrm{tr}}}$ and a sequence of information states $(\boldsymbol{s}, \boldsymbol{\tau})_n$ , $n=1,2,\ldots$, such that $\psi_{\boldsymbol{P}}((\boldsymbol{s}, \boldsymbol{\tau})_n)$ converges to $\psi^*(G_l; \boldsymbol{P})$ but there exists infinitely many $n$'s for which $O((\boldsymbol{s}, \boldsymbol{\tau})_n; \boldsymbol{P}) \neq O((\boldsymbol{s}, \boldsymbol{\tau})_{n+1}; \boldsymbol{P})$.

For simplicity of analysis, we focus on the following version of Assumption \ref{s:urestls:assump:esuboptimality}, although our results in Section \ref{single:sec:strong-analysis} will also hold when Assumption \ref{s:urestls:assump:esuboptimality} is true.

\begin{assumption} \label{s:urestls:assump:esuboptimality2}
There exists $\tau_{\textrm{tr}} \in \mathbb{N}$ such that for any $G_l \in {\cal G}_{\tau_{\textrm{tr}}}$, a single action is optimal for $\psi^*(G_l ; \boldsymbol{P})$.
\end{assumption}

\newt{In the next lemma, we show that when the set of transition probability matrices, i.e., $\boldsymbol{P}$, is drawn from a continuous distribution with a bounded density function (which is unknown to the player), before the play begins, then Assumption \ref{s:urestls:assump:esuboptimality2} will hold with probability one. In other words, the set of the set of transition probability matrices for which Assumption \ref{s:urestls:assump:esuboptimality2} does not hold is a measure zero subset of the set $\boldsymbol{\Xi}$.

\begin{lemma} \label{lemma:measurezero}
Let $f_{\boldsymbol{\Xi}}(.)$ be the density function of the distribution from which $\boldsymbol{P}$ is drawn as an instance of the bandit problem. Let $\sup_{\boldsymbol{P} \in \boldsymbol{\Xi}} f_{\boldsymbol{\Xi}}(\boldsymbol{P}) \leq f_{\max} < \infty$. Then, Assumption \ref{s:urestls:assump:esuboptimality2} holds with probability 1. In other words, when $w$ is a realization of transition probability matrices, and ${\cal A}$ be the event that Assumption \ref{s:urestls:assump:esuboptimality2} holds, then $P(w \in {\cal A}) =1 $.
\end{lemma}
\begin{proof}
If Assumption \ref{s:urestls:assump:esuboptimality2} does not hold for some $\boldsymbol{P} \in \boldsymbol{\Xi}$, this means that there exists $G_l \in {\cal G}_{\tau_{\textrm{tr}}}$ such that for some distinct arms $u$ and $u'$ we have
\begin{align}
 \bar{r}_{\boldsymbol{P}}(\boldsymbol{s}_l,u) + \sum_{y \in S^u} V_{\boldsymbol{P}}(\boldsymbol{s}_l,y,u) h_{\boldsymbol{P}}(T_{\boldsymbol{P}}(\boldsymbol{s}_l,y,u)) =  \bar{r}(\boldsymbol{s}_l,u') + \sum_{y \in S^{u'}} V_{\boldsymbol{P}}(\boldsymbol{s}_l,y,u') h_{\boldsymbol{P}}(T_{\boldsymbol{P}}(\boldsymbol{s}_l,y,u')), \label{eqn:continuation}
\end{align}
where $\bar{r}_{\boldsymbol{P}}(\boldsymbol{s}_l,u)$, $V_{\boldsymbol{P}}(\boldsymbol{s}_l,y,u)$ and $h_{\boldsymbol{P}}(T_{\boldsymbol{P}}(\boldsymbol{s}_l,y,u))$ are equivalents of the expressions given in Section \ref{single:sec:strong-aroe} for the case when we derive the optimal policy with respect to the countable information state instead of the belief state, where $s_l$ is the center information state of $G_l$. 

In order to have $P(w \in {\cal A}) < 1$, the integral of $f_{\boldsymbol{\Xi}}(.)$ over the subset of $\boldsymbol{\Xi}$ for which Assumption \ref{s:urestls:assump:esuboptimality2} does not hold should be positive. This means that there should be at least one $\boldsymbol{P}$ for which Assumption \ref{s:urestls:assump:esuboptimality2} does not hold for the $\epsilon$ neighborhood of $\boldsymbol{P}$, for some $\epsilon>0$. Next, we will show that for any $\epsilon>0$, there exists $\boldsymbol{P}'$ such that $||\boldsymbol{P}' - \boldsymbol{P} || \leq \epsilon$ and Assumption \ref{s:urestls:assump:esuboptimality2} holds for $\boldsymbol{P}'$. 
Let $ \epsilon_{p,\min} = \min_{k \in {\cal K}, i,j \in S^k} p^k_{i,j}$ and $\epsilon_s = \min\{ \epsilon/(2S_{\max}), \epsilon_{p, \min}/2 \}$.  For an arm $k$ let $\bar{x}_k$ be the the state with the highest reward and $\underline{x}_k$ be state with the lowest reward. We form $\boldsymbol{P}'$ from $\boldsymbol{P}$ as follows: For all arms other than $u$ and $u'$, state transition probabilities of $\boldsymbol{P}'$ is the same as $\boldsymbol{P}$. For arm $u$, let $(p')^u_{j,\underline{x}_u} = (p)^u_{j,\underline{x}_u} - \epsilon_s$ and $(p')^u_{j,\bar{x}_u} = (p)^u_{j,\bar{x}_u} + \epsilon_s$ for all $j \in S^u$. For arm $u'$ let $(p')^{u'}_{j,\bar{x}_{u'}} = (p)^{u'}_{j,\bar{x}_{u'}} - \epsilon_s$ and $(p')^{u'}_{j,\underline{x}_{u'}} = (p)^{u'}_{j,\underline{x}_{u'}} + \epsilon_s$ for all $j \in S^{u'}$. By construction we have $||\boldsymbol{P}' - \boldsymbol{P} || \leq \epsilon$. 

Due to the special assignment of probabilities in $\boldsymbol{P}'$, for any information state $(\boldsymbol{s}, \boldsymbol{\tau})$, we have $\bar{r}_{\boldsymbol{P}'}((\boldsymbol{s}, \boldsymbol{\tau}),u) > \bar{r}_{\boldsymbol{P}}((\boldsymbol{s}, \boldsymbol{\tau}),u)$ and  $\bar{r}_{\boldsymbol{P}'}((\boldsymbol{s}, \boldsymbol{\tau}),u') < \bar{r}_{\boldsymbol{P}}((\boldsymbol{s}, \boldsymbol{\tau}),u')$. Since arms evolve independently of each other, the continuation value of choosing arm $u$ in information state $(\boldsymbol{s}, \boldsymbol{\tau})$ under $\boldsymbol{P}'$ is greater than or equal to the continuation value of choosing arm $u$ under $\boldsymbol{P}$, while the continuation value of choosing arm $u'$ in information state $(\boldsymbol{s}, \boldsymbol{\tau})$ under $\boldsymbol{P}'$  is less than or equal to the continuation value of choosing arm $u'$ under $\boldsymbol{P}$. These and (\ref{eqn:continuation}) together implies that 
\begin{align*}
 \bar{r}_{\boldsymbol{P}'}(\boldsymbol{s}_l,u) + \sum_{y \in S^u} V_{\boldsymbol{P}'}(\boldsymbol{s}_l,y,u) h_{\boldsymbol{P}'}(T_{\boldsymbol{P}'}(\boldsymbol{s}_l,y,u)) >  \bar{r}(\boldsymbol{s}_l,u') + \sum_{y \in S^{u'}} V_{\boldsymbol{P}'}(\boldsymbol{s}_l,y,u') h_{\boldsymbol{P}'}(T_{\boldsymbol{P}'}(\boldsymbol{s}_l,y,u')).
\end{align*}

Note that for $\boldsymbol{P}'$, although the tie between arms $u$ and $u'$ is broken in favor of arm $u$ for center belief $s_l$, there can be another center belief $s_{l'}$ of $G_{l'} \in {\cal G}_{\tau_{\textrm{tr}}}$ for which two different arms $a$ and $a'$ became optimal in $\boldsymbol{P}'$, while only one arm is optimal for $s_{l'}$ in $\boldsymbol{P}$. If such a thing happens, then we can define a new transition probability matrix $\boldsymbol{P}''$ from $\boldsymbol{P}$ by subtracting and adding $\epsilon_s/2$ similar to the construction of $\boldsymbol{P}'$. This will both break the ties between $u$, $u'$ and $a$, $a'$ in favor of $u$ and $a$. Since the number of arms and the number of center beliefs in ${\cal G}_{\tau_{\textrm{tr}}}$ is finite, after repeating this procedure for a finite number of times we will find a $\tilde{\boldsymbol{P}}$ for which none of the center beliefs have more than one optimal arm such that $||\tilde{\boldsymbol{P}} - \boldsymbol{P} || \leq \epsilon$.
\end{proof}

Lemma \ref{lemma:measurezero} implies that even though there might exist some $\boldsymbol{P}$ for which Assumption \ref{s:urestls:assump:esuboptimality2} does not hold, for a perturbation of $\boldsymbol{P}$ Assumption \ref{s:urestls:assump:esuboptimality2} will hold.  Assumption \ref{s:urestls:assump:esuboptimality2} may not hold for some symmetric setting such as the case when there are identical arms. We illustrate this in the next example. 
\begin{example}
Consider 3 arms, with $S^k = \{0,1\}$ and $p^k_{01} := p_{01}$, $p^k_{10} := p_{10}$ and $ p_{01} +  p_{10}< 1$ for $k=\{1,2,3  \}$. Under this setting, it is shown in \cite{ahmad2009optimality} that the following myopic policy is optimal: At each time step, select the arm with the highest probability of being in state $1$. 
Consider the center belief $(\boldsymbol{\pi}^1, (p_{00}, p_{01}),\boldsymbol{\pi}^3)$ which occurs right after arm $2$ is played and state $0$ is observed. Since the arms are symmetric we have $\boldsymbol{\pi}^1 = \boldsymbol{\pi}^3$. Since the myopic policy is optimal it would be optimal to select either of arm $1$ and $3$ in the above center belief. Therefore Assumption \ref{s:urestls:assump:esuboptimality2} does not hold for this case.  
\end{example}
\begin{remark}
Although Assumption \ref{s:urestls:assump:esuboptimality2} does not hold in the example, 
note however the center belief $(\boldsymbol{\pi}^1, (p_{00}, p_{01}),\boldsymbol{\pi}^3)$ can only be reached under the following condition: (1) if arms 1 and 3 start with the same initial belief and are never played, or (2) arms 1 and 3 have not been played for infinitely long.  This is because only one arm can be played at a time, which results in asynchronous update of the belief states once an arm is played. Therefore in practice this center belief will never be reached in finite time.  
Therefore, in reality AREP never computes the values of its indices at this center belief. But AREP might be required to compute the value of its indices near this center belief.
\end{remark}
}

\begin{corollary} \label{s:urestls:corr:esuboptimality2}
Let $\tau_{\textrm{tr}} \in \mathbb{N}$ be the minimum integer such that Assumption \ref{s:urestls:assump:esuboptimality2} holds. Then, there exists $\bar{\epsilon}>0$, depending on $\tau_{\textrm{tr}}$, such that for all $\epsilon \leq \bar{\epsilon}$ and any $\psi \in J_{l,\epsilon}$, a single action is optimal. Also for any $\epsilon > \bar{\epsilon}$, there exists a $J_{l,\epsilon}$ and $\psi, \psi' \in J_{l,\epsilon}$ such that different actions are optimal at $\psi$ and $\psi'$.
\end{corollary}
\begin{proof}
This result follows from Assumption \ref{s:urestls:assump:esuboptimality2} and Lemma \ref{lemma:subsetoptimal}.
\end{proof}

\begin{remark}\label{remark:measure1}
Although we do not know of a way to check if Assumption \ref{s:urestls:assump:esuboptimality2} holds given a set of transition probability matrices $\boldsymbol{P}$, we \rdec{conjecture} that it holds for a large set of $\boldsymbol{P}$s \rdec{for the following reason}. 
%
%
This is because the player's selection does not affect state transitions of the arms; it only affects the player's information state. 
Moreover, each arm evolves independently from each other.  If $\boldsymbol{P}$ is arbitrarily selected from $\boldsymbol{\Xi}$, and the state rewards $r^k_x$, $x \in S^k$ are arbitrarily selected from $[0, r_{\max}]$, then at any information state $(\boldsymbol{s}, \boldsymbol{\tau}) \in {\cal C}$, the probability that the reward distribution of two arms are the same will be zero.  We therefore claim that Assumption \ref{s:urestls:assump:esuboptimality2} holds with probability one if the arm rewards and $\boldsymbol{P}$ are chosen from the uniform distribution on $\boldsymbol{\Psi} \times [0, r_{\max}]$. In other words, the set of arm rewards and transition probabilities for which Assumption \ref{s:urestls:assump:esuboptimality2} does not hold is a measure zero subset of $\boldsymbol{\Psi} \times [0, r_{\max}]$.
\end{remark}

%% file: equivalentiid.tex
\subsection{Implications of Assumptions 1, 2 and 3 for the iid bandit problem}
\aug{
To facilitate the understanding of Assumptions 1, 2 and 3, we explain their meanings for the iid bandit problem which is a special case of the URBP. 
In the URBP we consider $S^k$ is finite for each $k \in {\cal K}$ and the arm rewards are deterministic functions of the states. Hence, for the iid bandit problem which is a special case of the URBP we assume that
the reward of arm $k$ is drawn independently from a distribution $Q^k$ over a finite set $S^k$.
Assumption 1 says that for any arm $k$ and any state $i \in S^k$, every other state $j \in S^k$ is reachable in one time slot. This assumption is automatically satisfied in the iid setting for any arm reward distribution $Q^k$.\footnote{If there is a state $j \in S^k$ for which probability of reaching that state from some state $i \in S^k$ in one step is zero, then due to the iid assumption, probability of reaching state $j$ in one step from any other state must be zero, and hence the probability that state $j$ appears at any time slot is zero. Such states can be discarded since they will not appear with probability one.}

Assumption 2 says that for any set of arm reward distributions $\boldsymbol{Q} = (Q^1, \ldots, Q^K)$, there exists $\delta>0$ such that for any other set of arm reward distributions $\tilde{\boldsymbol{Q}} = (\tilde{Q}^1, \ldots, \tilde{Q}^K)$ such that $|| Q^k - \tilde{Q}^k  ||_1 \leq \delta$ for all $k \in {\cal K}$, the set of optimal actions (actions with the highest expected rewards) when the set of arm reward distributions is $\tilde{\boldsymbol{Q}} $ is the same subset of the set of optimal actions when the set of arm reward distributions is $\boldsymbol{Q}$. 
For the iid problem expected reward of arm $k$ under $Q^k$ is given as
\begin{align}
\mu_k(Q^k) := \sum_{x \in S} x Q^k(x) .     \notag
\end{align}
Let 
\begin{align}
{\cal K}^*(\boldsymbol{Q}) := \argmax_{k \in {\cal K}} \mu_k(Q^k)   ,   \notag
\end{align}
be the set of optimal arms under $\boldsymbol{Q}$,
$\mu^*(\boldsymbol{Q}) := \max_{k \in {\cal K}} \mu_k(Q^k)$ and
\begin{align}
\Delta(\boldsymbol{Q}) =  \mu^*(\boldsymbol{Q}) - \max_{k \in {\cal K} - {\cal K}^*(\boldsymbol{Q}) }  \mu_k(Q^k). \notag
\end{align}
Clearly if $|\mu_k(Q^k) - \mu_k(\tilde{Q}^k)| <\Delta(\boldsymbol{Q})/2 $ for all $k \in {\cal K}$, we have ${\cal K}^*(\tilde{\boldsymbol{Q}}) \subset {\cal K}^*(\boldsymbol{Q})$.
Assuming that states (arm rewards) are in $[0,1]$, this holds when 
\begin{align}
|Q^k(x) - \tilde{Q}^k(x)| \leq \Delta(\boldsymbol{Q})/ (2 S_{\max}) ,  \label{eqn:valass2iid}
\end{align}
for all $k \in {\cal K}$ and $x \in S^k$.
Recall that the belief vector $\psi$ is an $\boldsymbol{S}$ dimensional vector whose $\boldsymbol{x}$th component corresponds to the probability that the joint state is $\boldsymbol{x}$. For the iid setting the belief at time $t$ is a constant, since the reward distribution at time $t+1$ is independent from the reward distribution at time $t$.
Hence under $\boldsymbol{Q}$ the $\boldsymbol{x}$th component of the belief vector is equal to
\begin{align}
\psi_{\boldsymbol{x}}(\boldsymbol{Q})  = \prod_{k \in {\cal K}} Q^k(x_k) .   \notag
\end{align}
Let $\epsilon = \Delta(\boldsymbol{Q})/ (2 S_{\max}^K )$.
Since
\begin{align}
||\psi(\boldsymbol{Q})  - \psi(\tilde{\boldsymbol{Q}}) ||_1
& = \sum_{\boldsymbol{x} \in \boldsymbol{S}} \left|\prod_{k \in {\cal K}} Q^k(x_k)  - \prod_{k \in {\cal K}} \tilde{Q}^k(x_k) \right| ,
\end{align}
$||\psi(\boldsymbol{Q})  - \psi(\tilde{\boldsymbol{Q}}) ||_1 \leq \epsilon $ implies that $| \psi_{\boldsymbol{x}}(\boldsymbol{Q})   - \psi_{\boldsymbol{x}}(\tilde{\boldsymbol{Q}}) | \leq \epsilon$ for all $\boldsymbol{x} \in \boldsymbol{S}$.
We also have for any $x \in S^k$
\begin{align}
Q^k(x) = \sum_{\boldsymbol{x} \in \boldsymbol{S}: x_k = x} \psi_{\boldsymbol{x}}(\boldsymbol{Q})    .    \notag
\end{align}
Hence we have for any $k \in {\cal K}$, $x \in S^k$
\begin{align}
 |Q^k(x) - \tilde{Q}^k(x)| \leq  \sum_{\boldsymbol{x} \in \boldsymbol{S}: x_k = x}  | \psi_{\boldsymbol{x}}(\boldsymbol{Q}) - \psi_{\boldsymbol{x}}(\tilde{\boldsymbol{Q}}) |    \leq  \Delta(\boldsymbol{Q})/ (2 S_{\max} ) ,\notag
\end{align}
which is equivalent to (\ref{eqn:valass2iid}).
Hence, ${\cal K}^*(\tilde{\boldsymbol{Q}}) \subset {\cal K}^*(\boldsymbol{Q})$ holds in the iid setting with $\epsilon = \Delta(\boldsymbol{Q})/ (2 S_{\max}^K )$ (given in Assumption 2) for any set of arm reward distributions $\boldsymbol{Q}$.
However, Assumption 2 requires that for any $\boldsymbol{Q}$, ${\cal K}^*(\tilde{\boldsymbol{Q}})$ must be the same for all $\tilde{\boldsymbol{Q}}$ such that $||\psi(\boldsymbol{Q})  - \psi(\tilde{\boldsymbol{Q}}) ||_1 \leq \epsilon$. This can only hold in the iid setting when there is a unique optimal arm for $\boldsymbol{Q}$, which is given by Assumption 3.

The iid setting in which there are multiple optimal arms also exhibits a special symmetric structure. As we proved in Lemma 3, such symmetric structures appear only on a measure zero subset of the set of arm reward distributions. 
}

%% file: strong_analysis.tex
In this section we show that when $\boldsymbol{P}$ is such that Assumptions \ref{s:urestls:assump:1} and \ref{s:urestls:assump:esuboptimality2} hold and when using AREP with $f(t) = L \log t$ with $L$ sufficiently large (i.e., $L \geq C(\boldsymbol{P})$, a constant dependent  on $\boldsymbol{P}$), the regret due to explorations is logarithmic in time, while the regret due to all other terms are finite, independent of $t$. 
Note that since the player does not know $\boldsymbol{P}$, it cannot know how large it should chose $L$. For simplicity we assume that the player starts with an $L$ that is large enough without knowing $C(\boldsymbol{P})$. We also prove a near-logarithmic regret result in Section \ref{sec:extensions_explore} when the player sets $f(t) = L(t) \log t$, where $L(t)$ is a positive increasing function over time such that $\lim_{t \rightarrow \infty} L(t) = \infty$.

In what follows, we first provide an upper bound on the regret as the summation of a number of components. We then proceed to bound these individual terms separately. 
%
%
%

%% file: strong_regretbound.tex
For any admissible policy $\alpha$, the regret with respect to the optimal $T$ horizon policy is given in (\ref{eqn:probform1}), which we restate below:
\begin{align*}
\sup_{\gamma \in \Gamma} \left( E^{\boldsymbol{P}}_{\psi_0, \gamma}\left[\sum_{t=1}^T  r^{\gamma(t)}(t)\right]\right) - E^{\boldsymbol{P}}_{\psi_0, \alpha}\left[\sum_{t=1}^T r^{\alpha(t)}(t)\right].
\end{align*}
We first derive the regret with respect to the optimal policy as a function of the number of suboptimal plays. Before proceeding, we introduce expressions to compactly represent the RHS of the AROE. Let
\begin{align}
{\cal L}(\psi,u,h,\boldsymbol{P}) &:= \bar{r}(\psi,u) + (V(\psi,.,u) \bullet h(T_{\boldsymbol{P}}(\psi,.,u))) \notag \\
{\cal L}^*(\psi,\boldsymbol{P}) &:= \max_{u \in U} {\cal L}(\psi,u,h_{\boldsymbol{P}},\boldsymbol{P}) \notag \\
\Delta(\psi,u;\boldsymbol{P}) &:= {\cal L}^*(\psi,\boldsymbol{P}) - {\cal L}(\psi,u,h_{\boldsymbol{P}},\boldsymbol{P}) ~,  \label{s:urestls:eqn:suboptimality}
\end{align}
where the last one denotes the degree of suboptimality of action $u$ at information state $\psi$ when the set of transition probability matrices is $\boldsymbol{P}$. 

From Proposition 1 in \cite{burnetas1997optimal}, we have for all $\gamma \in \Gamma$
\begin{align}
R^\gamma_{(\psi_0;\boldsymbol{P})}(T) = \sum_{t=1}^{T} E^{\boldsymbol{P}}_{\psi_0,\gamma}[\Delta(\psi_t, U_t; \boldsymbol{P})] \rdec{+ \bar{C}_{\boldsymbol{P}}}, \label{eqn:defnregret}
\end{align}
for some constant $\bar{C}_{\boldsymbol{P}}$ dependent on $\boldsymbol{P}$, and $U_t$ is the random variable denoting the arm selected by the player at time $t$ which depends on the policy used. 
We have used the subscript $(\psi_0;\boldsymbol{P})$ to denote the dependence of regret on the initial belief and the transition probabilities. We assume that initially all the arms are sampled once thus the initial belief is $\psi_0 = \psi_{\boldsymbol{P}}((\boldsymbol{s}_0, \boldsymbol{\tau}_0))$. 
For the true set of transition probability matrices $\boldsymbol{P}$, let $\tau_{\textrm{tr}}$ and $\bar{\epsilon}$ be the numbers given in Corollary \ref{s:urestls:corr:esuboptimality2}.
Denote the $\bar{\epsilon}$-extension of the set $G_l \in {\cal G}_{\tau_{\textrm{tr}}}$ by $J_{l, \bar{\epsilon}}$. Note that at any $t$, the belief $\psi_t \in J_{l, \bar{\epsilon}}$ for some $l$. For simplicity of notation, when $\bar{\epsilon}$ is clear from the context we will re-write $J_{l, \bar{\epsilon}}$ as $J_l$.
Let
\begin{align*}
\bar{\Delta}(J_l, u; \boldsymbol{P}) := \sup_{\psi \in J_l} \Delta(\psi, u; \boldsymbol{P}).
\end{align*}
Note that if $U_t \in O(\psi_t;\boldsymbol{P})$ then $\Delta(\psi_t, U_t; \boldsymbol{P})=0$; otherwise $U_t \notin O(\psi_t; \boldsymbol{P})$, and then $\Delta(\psi_t, U_t; \boldsymbol{P}) \leq \bar{\Delta}(J_l, U_t; \boldsymbol{P})$ with probability one. Let
\begin{align*}
N_T(J_l,u) := \sum_{t=1}^{T} I(\psi_t \in J_l, U_t =u).
\end{align*}
%
\begin{lemma} \label{s:urestls:eqn:regret}
For any admissible policy $\gamma$, 
\begin{align*}
R^\gamma_{(\psi_0;\boldsymbol{P})}(T) \leq \sum_{l=1}^{A(\tau_{\textrm{tr}})} \sum_{u \notin O(J_l;\boldsymbol{P})} E^{\boldsymbol{P}}_{\psi_0,\gamma} [N_T(J_l,u)] \bar{\Delta}(J_l, u; \boldsymbol{P})+ \bar{C}_{\boldsymbol{P}}. 
\end{align*}
\end{lemma}
\begin{proof}
\begin{align}
R^\gamma_{(\psi_0;\boldsymbol{P})}(T) 
&\leq  \sum_{t=1}^{T} E^{\boldsymbol{P}}_{\psi_0,\gamma} \left[\sum_{l=1}^{A(\tau_{\textrm{tr}})} \sum_{u \notin O(J_l;\boldsymbol{P})} I(\psi_t \in J_l, U_t =u) \bar{\Delta}(J_l, u; \boldsymbol{P}) \right] + \bar{C}_{\boldsymbol{P}} \notag \\
&= \sum_{l=1}^{A(\tau_{\textrm{tr}})} \sum_{u \notin O(J_l;\boldsymbol{P})} E^{\boldsymbol{P}}_{\psi_0,\gamma} \left[\sum_{t=1}^{T} I(\psi_t \in J_l, U_t =u) \right] \bar{\Delta}(J_l, u; \boldsymbol{P}) + \bar{C}_{\boldsymbol{P}} \notag \\
&= \sum_{l=1}^{A(\tau_{\textrm{tr}})} \sum_{u \notin O(J_l;\boldsymbol{P})} E^{\boldsymbol{P}}_{\psi_0,\gamma} [N_T(J_l,u)] \bar{\Delta}(J_l, u; \boldsymbol{P}) + \bar{C}_{\boldsymbol{P}}. \notag
\end{align}
\end{proof}

Now consider AREP, which is denoted by $\alpha$. We will upper bound $N_T(J_l,u)$ for suboptimal actions $u$ by a sum of expressions which we will then bound individually. Let ${\cal E}_t$ be the event that AREP is in an exploitation step at time $t$ and 
%
${\cal F}_t(\epsilon) := \left\{ \left\|\hat{h}_t - h_{\boldsymbol{P}}\right\|_{\infty} \leq \epsilon \right\}$.
%
For an event ${\cal F}$, denote its complement by ${\cal F}^c$. 
\aug{For any $\epsilon > 0$}, consider the following random variables which count the number of times a certain event has happened by time $T$ (the dependence on $T$ is dropped from the notation for convenience).
\begin{align}
D_{1,1}(\epsilon,J_l,u) &:= \sum_{t=1}^{T} I \left(\hat{\psi}_t \in J_l, U_t=u, {\cal E}_t, {\cal F}_t(\epsilon) \right), \notag \\
D_{1,2}(\epsilon) &:= \sum_{t=1}^{T} I({\cal E}_t, {\cal F}^c_t(\epsilon)), \notag \\
D_1(\epsilon,J_l,u) &:= D_{1,1}(\epsilon,J_l,u) + D_{1,2}(\epsilon), \notag \\
D_{2,1}(\epsilon) &:= \sum_{t=1}^{T} I(|| \psi_t-\hat{\psi}_t ||_1 > \epsilon, {\cal E}_t), \notag \\
D_{2,2}(\epsilon, J_l) &:= \sum_{t=1}^{T} I(|| \psi_t-\hat{\psi}_t ||_1 \leq \epsilon,   
\hat{\psi}_t \notin J_l,\psi_t \in J_l, {\cal E}_t), \notag \\
D_2(\epsilon, J_l) &:= D_{2,1}(\epsilon) + D_{2,2}(\epsilon, J_l). \notag
\end{align}
\begin{lemma} \label{s:urestls:lemma:regretbound}
For any $\boldsymbol{P}$ satisfying Assumption \ref{s:urestls:assump:esuboptimality2}, we have 
\begin{align}
E^{\boldsymbol{P}}_{\psi_0,\gamma} [N_T(J_l,u)] &\leq E^{\boldsymbol{P}}_{\psi_0,\gamma}[D_1(\epsilon,J_l,u)] 
+ E^{\boldsymbol{P}}_{\psi_0,\gamma}[D_2(\epsilon, J_l)]   + E^{\boldsymbol{P}}_{\psi_0,\gamma} \left[\sum_{t=1}^{T} I({\cal E}^c_t) \right]. \label{s:urestls:eqn:regretbound}
\end{align}
\end{lemma}
\begin{proof} 
\begin{align*}
&N_T(J_l,u) = \sum_{t=1}^{T} (I(\psi_t \in J_l, U_t =u, {\cal E}_t) 
+ I(\psi_t \in J_l, U_t =u, {\cal E}^c_t) )\\
&\leq \sum_{t=1}^{T} I(\psi_t \in J_l, \hat{\psi}_t \in J_l, U_t =u, {\cal E}_t) 
+ \sum_{t=1}^{T} I(\psi_t \in J_l, \hat{\psi}_t \notin J_l, U_t =u, {\cal E}_t) \\
&+  \sum_{t=1}^{T} I({\cal E}^c_t) \\
&\leq \sum_{t=1}^{T} I(\hat{\psi}_t \in J_l, U_t=u, {\cal E}_t) 
+ \sum_{t=1}^{T} I(\psi_t \in J_l, \hat{\psi}_t \notin J_l, {\cal E}_t) +  \sum_{t=1}^{T} I({\cal E}^c_t) \\
&\leq D_{1,1}(\epsilon,J_l,u)  + D_{1,2}(\epsilon)
+ D_{2,1}(\epsilon) + D_{2,2}(\epsilon, J_l) \\
&+  \sum_{t=1}^{T} I({\cal E}^c_t).
\end{align*}
The result follows from taking the expectation on both sides.
\end{proof}

%% file: strong_analysis_explorations.tex
The following lemma bounds the number of explorations by time $T$.

\begin{lemma} \label{s:urestls:lemma:subsection1}
\begin{align}
E^{\boldsymbol{P}}_{\psi_0,\alpha} \left[ \sum_{t=1}^{T} I({\cal E}^c_t) \right] \leq \left(\sum_{k=1}^K |S^k| \right) L \log T (1+T_{\max}), \label{s:urestls:eqn:subsection1}
\end{align}
where $T_{\max} = \max_{k \in {\cal K}, i,j \in S^k} E[T^k_{ij}] +1$ and $T^k_{ij}$ is the hitting time of state $j$ of arm $k$ starting from state $i$ of arm $k$. Since all arms are ergodic $E[T^k_{ij}]$ is finite for all $k \in {\cal K}, i,j \in S^k$.
\end{lemma}
\begin{proof} 
Assume that state $i$ of arm $k$ is {\em under-sampled}, i.e., $C^k_i(t) < L \log t$. Since arms are  ergodic, if the player keeps playing arm $k$, the expected number of time steps until a transition out of state $i$ is observed is at most $(1+T_{\max})$. If by time $T$, transitions out of state $i$ of arm $k$ is observed at least $L \log T$ times, for all states $i$ of all arms $k$, then the player will not explore at time $T$. Therefore there can be at most $\sum_{k=1}^K \sum_{i \in S^k} L \log T$ such transitions by time $T$ that take place in an exploration step. In the worst-case each of these transitions takes $(1+T_{\max})$ expected time steps. 
%
\end{proof}

%% file: strong_para2.tex
\rdec{We begin with the following lemma, based on the Chernoff-Hoeffding bound, that shows that the probability that an estimated transition probability is significantly different from the true transition probability given AREP is in an exploitation phase is very small.} 

\rdec{
\begin{lemma} \label{s:urestls:lemma:largedev}
For any $\epsilon' > 0$, for a player using AREP with constant $L \geq 1/ (\epsilon')^2$, we have
\begin{align*}
P\left(|\hat{p}^k_{ij,t} - p^k_{ij}| > \epsilon', {\cal E}_t \right) 
:= P\left( \{|\hat{p}^k_{ij,t} - p^k_{ij}| > \epsilon' \} \cap {\cal E}_t \right) 
\leq \frac{ 2}{t^2},
\end{align*}
for all $t>0$, $i, j \in S^k$, $k \in {\cal K}$.
\end{lemma}
\begin{proof}
See Appendix \ref{s:urestls:app:lemmalargedev}.
\end{proof}
} 
 
We next bound $E^{\boldsymbol{P}}_{\psi_0,\alpha}[ D_{1,1}(\epsilon,J_l,u)]$ for any suboptimal $u$. Let 
\begin{align*}
\underline{\Delta}(J_l;\boldsymbol{P}) &:= \min_{\psi \in J_l, u \notin O(J_l; \boldsymbol{P})} \Delta (\psi, u; \boldsymbol{P}).
\end{align*}
By Corollary \ref{s:urestls:corr:esuboptimality2}, $\underline{\Delta}(J_l;\boldsymbol{P})>0$ for all $l=1,\ldots, A(\tau_{\textrm{tr}})$. Let 
\begin{align}
\underline{\Delta} := \min_{l=1,\ldots, A(\tau_{\textrm{tr}})} \underline{\Delta}(J_l;\boldsymbol{P}). \label{eqn:underlinedelta}
\end{align}
In the following lemma we show that when the transition probability estimates are sufficiently accurate and the estimated solution to the AROE is sufficiently close to the true solution, a suboptimal action cannot be chosen by the player.
\begin{lemma}  \label{s:urestls:lemma:subsec30}
Let $\delta_e>0$ (depending on $\tau_{\textrm{tr}}$) be the greatest real number such that 
\begin{align*}
|| \hat{\boldsymbol{P}}_t - \boldsymbol{P}||_1 < \delta_e &\Rightarrow 
\left| {\cal L}(\psi, u, h_{\boldsymbol{P}}, \boldsymbol{P} ) - {\cal L}(\psi, u, h_{\boldsymbol{P}}, \hat{\boldsymbol{P}}_t ) \right| \leq \underline{\Delta}/4, 
\end{align*}
for all $\psi \in \boldsymbol{\Psi}$. Such $\delta_e$ exists because $T_{\boldsymbol{P}}(\psi,y,u)$ is continuous in $\boldsymbol{P}$, and $h_{\boldsymbol{P}}(\psi)$ is continuous in $\psi$.
Then for a player using AREP with \aug{$L \geq K^2 S^4_{\max}/ \delta_e^2$}, for any suboptimal action $u \notin O(J_l; \boldsymbol{P})$, we have
\begin{align*}
E^{\boldsymbol{P}}_{\psi_0,\alpha}[ D_{1,1}(\epsilon,J_l,u)] \leq \aug{2 K S^2_{\max} \beta},
\end{align*}
for $\epsilon < \underline{\Delta}/4$, where $\beta= \sum_{t=1}^{\infty} 1/t^{2}$. 
\end{lemma}
\begin{proof}
See Appendix \ref{s:urestls:app:lemma:subsec30}.
\end{proof}

Next we bound $E^{\boldsymbol{P}}_{\psi_0,\alpha} [D_{1,2}(\epsilon)]$. To do this we introduce the following lemma which implies that when the estimated transition probabilities get close to the true values, the difference between the solutions to the AROE based on the estimated and true values diminishes. 
\begin{lemma} \label{s:urestls:lemma:subsec34}
For any $\epsilon>0$, there exists $\varsigma(\epsilon) >0$ depending on $\epsilon$ such that if $\left\|P^k - \hat{P}^k\right\|_1 < \varsigma(\epsilon), \forall k \in {\cal K}$ then $\left\|h_{\boldsymbol{P}} - h_{\hat{\boldsymbol{P}}} \right\|_\infty < \epsilon$.
\end{lemma}

\begin{proof}
See Appendix \ref{s:urestls:app:subsec34}.
\end{proof}
The following lemma bounds $E^{\boldsymbol{P}}_{\psi_0,\alpha} [D_{1,2}(\epsilon)]$. 
\begin{lemma} \label{s:urestls:lemma:subsec35}
For any $\epsilon>0$, let $\varsigma(\epsilon) > 0$ be such that Lemma \ref{s:urestls:lemma:subsec34} holds. Then for a player using AREP with \aug{$L \geq S^4_{\max}/ (\varsigma(\epsilon))^2$}, we have
\begin{align}
E^{\boldsymbol{P}}_{\psi_0,\alpha} [D_{1,2}(\epsilon)] \leq 2K S^2_{\max}  \beta . 
\end{align}
\end{lemma}
\begin{proof} 
See Appendix \ref{s:urestls:app:lemma:subsec35}.
\end{proof}

%% file: strong_para1.tex
\begin{lemma} \label{s:urestls:lemma:subsec21}
For a player using AREP with exploration constant \aug{$L \geq (K S^2_{\max} |S^1| \ldots |S^K| C_1(\boldsymbol{P}))^2/ \epsilon^2$}, we have
\begin{align}
E^{\boldsymbol{P}}_{\psi_0,\alpha}[D_{2,1}(\epsilon)] \leq 2K S^2_{\max} \beta, \notag
\end{align}
where $C_1(\boldsymbol{P}) = \max_{k \in {\cal K}} C_1(P^k,\infty)$ and $C_1(P^k, t)$ is a constant that can be found in Lemma \ref{intro:lemma:ergodic} in Appendix \ref{s:urestls:app:lemmalargedev1}. 
\end{lemma}
\begin{proof} 
See Appendix \ref{s:urestls:app:lemma:subsec21}.
\end{proof}

Next we will bound $E^{\boldsymbol{P}}_{\psi_0,\alpha}[D_{2,2}(\epsilon, J_l)]$. 

\begin{lemma} \label{s:urestls:lemma:subsec22}
Let $\tau_{\textrm{tr}}$ be such that Assumption \ref{s:urestls:assump:esuboptimality2} holds. Then for $\epsilon < \bar{\epsilon}/2$, where $\bar{\epsilon}$ is given in Corollary \ref{s:urestls:corr:esuboptimality2}, $E^{\boldsymbol{P}}_{\psi_0,\alpha}[D_{2,2}(\epsilon,J_l)]=0, l=1,\ldots, A(\tau_{\textrm{tr}})$.
\end{lemma}
\begin{proof} 
By Corollary \ref{s:urestls:corr:esuboptimality2}, any $\psi_t \in J_l$ is at least $\bar{\epsilon }$ away from the boundary of $J_l$. Thus given $\hat{\psi}_t$ is at most $\epsilon$ away from $\psi_t$, it is at least $\bar{\epsilon}/2$ away from the boundary of $J_l$.
\end{proof} 

%% file: strong_para3.tex
\begin{theorem} \label{s:urestls:theorem:main}
\aug{
Assume that \revq{Assumptions \ref{s:urestls:assump:1} and \ref{s:urestls:assump:esuboptimality2} are true}. Let $\tau_{\textrm{tr}}$ be the minimum threshold, and $\bar{\epsilon}$ be the number given in Corollary \ref{s:urestls:corr:esuboptimality2}. 
Let
\begin{align}
\epsilon = \min \left\{ \frac{\underline{\Delta}}{8}, \frac{\bar{\epsilon}}{4}  \right\}  ,    \notag
\end{align}
where $\underline{\Delta}$ is given in (\ref{eqn:underlinedelta}). Let
\begin{align}
C(\boldsymbol{P}) := \max \left\{ \frac{K^2 S^4_{\max}}{\delta^2_e} , \frac{S^4_{\max}}{\varsigma(\epsilon)^2}, 
\frac{ ( K S^2_{\max} |S^1| \ldots |S^K| C_1(\boldsymbol{P})  )^2     }{\epsilon^2}   \right\}   ,       \notag
\end{align}
where $\delta^2_e >0$ is the constant given in Lemma \ref{s:urestls:lemma:subsec30}, $\varsigma(\epsilon)$ is the constant given in Lemma \ref{s:urestls:lemma:subsec35} and $C_1(\boldsymbol{P})$ is the constant given in Lemma \ref{s:urestls:lemma:subsec21}.
For a player using AREP with $L \geq C(\boldsymbol{P})$, 
for any arm (action) $u \in U$ which is suboptimal for the belief vectors in $J_l$, we have
\begin{align}
E^{\boldsymbol{P}}_{\psi_0, \alpha} [N_T(J_l,u)] \leq  \left(\sum_{k=1}^K |S^k| \right) L \log T (1+T_{\max}) + 6 K S^2_{\max} \beta ~. \notag
\end{align}
}
Therefore,
\begin{align}
R^{\alpha}_{\psi_0; \boldsymbol{P}}(T) &\leq \left(\left(\sum_{k=1}^K |S^k| \right)L \log T (1+T_{\max}) + 6 K S^2_{\max} \beta\right) \notag \\
&\times \sum_{l=1}^{A(\tau_{\textrm{tr}})} \sum_{u \notin O(J_l;\boldsymbol{P})}  \bar{\Delta}(J_l, u; \boldsymbol{P}) + \bar{C}_{\boldsymbol{P}}. \notag
\end{align}
When the arm rewards are in $[0,1]$, strong regret at time $T$ given as $R^{\alpha}_{\psi_0; \boldsymbol{P}}(T)$ can also be upper bounded by
\begin{align*}
\left(\left(\sum_{k=1}^K |S^k| \right)L \log T (1+T_{\max}) +6 K S^2_{\max}  \beta\right) (K A(\tau_{\textrm{tr}})) + \bar{C}_{\boldsymbol{P}} ~.
\end{align*}
\end{theorem}
\begin{proof} 
\aug{
In order for the bound in Lemma \ref{s:urestls:lemma:subsec30} to hold it is sufficient that $\epsilon < \underline{\Delta}/4$. In order for the bound in Lemma \ref{s:urestls:lemma:subsec22} to hold it is sufficient that
$\epsilon < \bar{\epsilon}/2$. These two conditions are satisfied when $\epsilon = \min \left\{ \frac{\underline{\Delta}}{8}, \frac{\bar{\epsilon}}{4}  \right\}$. 
In order for the bound in Lemma \ref{s:urestls:lemma:subsec30} to hold it is sufficient that $L \geq K^2 S^4_{\max}/ \delta_e^2$. Similarly the sufficient condition for Lemma \ref{s:urestls:lemma:subsec35} is $L \geq S^4_{\max}/ (\varsigma(\epsilon))^2$, and Lemma \ref{s:urestls:lemma:subsec21} is $L \geq (K S^2_{\max} |S^1| \ldots |S^K| C_1(\boldsymbol{P}))^2/ \epsilon^2$.
}

The regret bound follows from combining the results of Lemmas \ref{s:urestls:eqn:regret}, \ref{s:urestls:lemma:subsection1}, \ref{s:urestls:lemma:subsec30}, \ref{s:urestls:lemma:subsec35}, \ref{s:urestls:lemma:subsec21} and \ref{s:urestls:lemma:subsec22}. 
\end{proof} 

\begin{remark}
Our regret bound depends on $A(\tau_{\textrm{tr}})$. However, the player does not need to know the value of $\tau_{\textrm{tr}}$ for which Corollary \ref{s:urestls:corr:esuboptimality2} is true. It only needs to choose $L$ large enough so that the number of exploration steps is sufficient to ensure a bounded number of errors in exploitation steps. 
\aug{
In the next section we will propose an extension to AREP such that the player can achieve near-logarithmic in time regret without knowing the sufficient condition on $L$.
}
\end{remark}

\subsection{A comment on the worst-case regret bound of AREP}
\aug{
Theorem 1 gives a logarithmic in time regret bound for AREP. This bound depends on the true set of transition probability matrices $\boldsymbol{P}$ since the sufficient condition on $L$, $A(\tau_{\textrm{tr}})$ and $\bar{C}_{\boldsymbol{P}}$ given in Theorem 1 depend on $\boldsymbol{P}$. This type of regret bounds are called instance (distribution) dependent regret bounds. Bounds on regret that hold independent of $\boldsymbol{P}$ are called worst-case (distribution-free) bounds. There exists algorithms for URBP with $\tilde{O}(\sqrt{T})$ distribution-free regret bounds \cite{ortner2012regret}. However, it is an open question if an algorithm can achieve both $O(\log T)$ instance-dependent and  $\tilde{O}(\sqrt{T})$ distribution-free regret bound for URBP. 

Indeed, it is proven to be very difficult to find a general distribution-free regret bound for AREP. One of the reasons is that
the numbers $\epsilon$, $\varsigma(\epsilon)$, $C_1(\boldsymbol{P})$ and $\delta^2_e$ in Theorem 1 depend on the true set of transition probabilities $\boldsymbol{P}$ and rewards, but they don't have closed form expressions as functions of transition probabilities and rewards. This is due to the fact that existence of these constants are proven using the continuity property of the solutions to the average reward optimality equation in URBP. However, there is no analytical expression for the exact form of the solution in the theory of finite probabilistic systems \cite{platzman1980}. 
}

%% file: strong_extension_explore.tex
In this section, we consider an adaptive exploration function for AREP, by which the player can achieve near-logarithmic regret without knowing how large it should chose the exploration constant $L$, which depends on $\boldsymbol{P}$. 
%
%
First note that the analysis in Section \ref{single:sec:strong-analysis} holds when AREP is run with a sufficiently large exploration constant $L\geq C(\boldsymbol{P})$ such that in each exploitation step the estimated transition probabilities $\hat{\boldsymbol{P}}$ is close enough to $\boldsymbol{P}$ to guarantee that all regret terms in (\ref{s:urestls:eqn:regretbound}) is finite except that due to explorations. 
%
Practically, since the player does not know the transition probabilities initially, it may be unreasonable to assume that it can check if $L \geq C(\boldsymbol{P})$.  One possible solution is to assume that the player knows \newt{a compact set} $\tilde{\boldsymbol{\Xi}} \subset \boldsymbol{\Xi}$, the set of transition probability matrices where $\boldsymbol{P}$ lies in. 
If this is the case, then it can compute $\tilde{C} = \newt{\max}_{\tilde{\boldsymbol{P}} \in \tilde{\boldsymbol{\Xi}}} C(\tilde{\boldsymbol{P}})$, and choose $L > \tilde{C}$. 

In this section, we present another exploration function for AREP such that the player can achieve near-logarithmic regret even without knowing $C(\boldsymbol{P})$ or $\tilde{C}$. Let $f(t) = L(t) \log t$ where $L(t)$ is an increasing function such that $L(1)=1$ and $\lim_{t \rightarrow \infty} L(t) = \infty$. The intuition behind this exploration function is that after some time $T_0$, $L(t)$ will be large enough so that the estimated transition probabilities are sufficiently accurate, and the regret due to incorrect calculations becomes a constant independent of time. 

\begin{theorem}
When $\boldsymbol{P}$ is such that Assumptions \ref{s:urestls:assump:1} and \ref{s:urestls:assump:esuboptimality2} hold, if the player uses AREP with $f(t) = L(t) \log t$, for some increasing $L(t)$ such that $L(1)=1$ and $\lim_{t \rightarrow \infty} L(t) = \infty$, then there exists $\tau_{\textrm{tr}}(\boldsymbol{P}) > 0$, $T_0(L, \boldsymbol{P})>0$ such that the strong regret is upper bounded by
\begin{align*}
R^{\alpha}_{\psi_0; \boldsymbol{P}}(T) & \leq r_{\max} \left( T_0(L, \boldsymbol{P}) + \left(\sum_{k=1}^K |S^k| \right) L(T) \log T (1+T_{\max}) \right. \\
&\left. +6 K S^2_{\max}  \beta 
\left( \sum_{l=1}^{A(\tau_{\textrm{tr}})} \sum_{u \notin O(J_l;\boldsymbol{P})}  \bar{\Delta}(J_l, u; \boldsymbol{P})\right) \right)  \\ 
&\leq r_{\max} \left( T_0(L, \boldsymbol{P}) + \left(\sum_{k=1}^K |S^k| \right) L(T) \log T (1+T_{\max}) \right. \\
& \left. + 6 K S^2_{\max} \beta 
 (\tau_{\textrm{tr}})^M  \left(\sum_{k=1}^K |S^k| \right) \max_{l \in \{1,\ldots, A(\tau_{\textrm{tr}}) \}}  \bar{\Delta}(J_l, u; \boldsymbol{P})
 \right) + \bar{C}_{\boldsymbol{P}} ~. 
\end{align*}
\end{theorem}
\begin{proof}
The regret up to $T_0(L, \boldsymbol{P})$ can be at most $r_{\max} T_0(L, \boldsymbol{P})$. After $T_0(L, \boldsymbol{P})$, since $L(t) \geq C(\boldsymbol{P})$, transition probabilities at exploitation steps sufficiently accurate so that all regret terms in (\ref{s:urestls:eqn:regretbound}) except the regret due to explorations is finite. Since time $t$ is an exploration step whenever $C^k_i(t) < L(t) \log t$, the regret due to explorations is at most
%
$r_{\max} \left(\sum_{k=1}^K |S^k| \right) L(T) \log T (1+T_{\max})$.
%
\end{proof}

\begin{remark}\label{remark:unknowngap}
There is a tradeoff between choosing a rapidly increasing $L(t)$ or a slowly increasing $L(t)$. The regret of AREP up to time $T_0(L, \boldsymbol{P})$ is linear. Since $T_0(L, \boldsymbol{P})$ is a decreasing function in $L(t)$, a rapidly increasing $L(t)$ will have better performance when the time horizon is small. However, in terms of asymptotic performance as $T \rightarrow 
\infty$, $L(t)$ should be a slowly diverging sequence. For example if $L(t) = \log(\log t)$, then the asymptotic regret will be $O(\log(\log t) \log t )$. 
\end{remark}

%% file: strong_extension_finpar.tex
In this section we present a modified version of AREP and prove that it can achieve logarithmic regret without \rdec{Assumption \ref{s:urestls:assump:esuboptimality} or \ref{s:urestls:assump:esuboptimality2}} if the player knows the time horizon $T$. 
We call this variant AREP with finite partitions (AREP-FP).
 
Basically, AREP-FP takes as input the threshold or mixing time $\tau_{\textrm{tr}}$ and then forms the ${\cal G}_{\tau_{\textrm{tr}}}$ partition of the set of information states ${\cal C}$. At each exploitation step (time $t$) AREP-FP solves the estimated AROE based on the transition probability estimate $\hat{\boldsymbol{P}}_t$.  If the information state $(\boldsymbol{s}_t, \boldsymbol{\tau}_t) \in G_l$, the player arbitrarily picks an arm in $O^*(G_l; \hat{\boldsymbol{P}}_t)$, instead of picking an arm in $O(\psi_{\hat{\boldsymbol{P}}_t}((\boldsymbol{s}_t, \boldsymbol{\tau}_t)); \hat{\boldsymbol{P}}_t)$.
%
If the arm selected by the player is indeed in $O(\psi_{\boldsymbol{P}}((\boldsymbol{s}_t, \boldsymbol{\tau}_t)); \boldsymbol{P})$, then it ends up playing optimally at that time step. Else if the selected arm is in $O^*(G_l; \boldsymbol{P})$ but not in $O(\psi_{\boldsymbol{P}}((\boldsymbol{s}_t, \boldsymbol{\tau}_t)); \boldsymbol{P})$ such that $(\boldsymbol{s}_t, \boldsymbol{\tau}_t) \in G_l$, then it plays near-optimally. 
Finally, it plays suboptimally if the selected arm is neither in $O(\psi_{\boldsymbol{P}}((\boldsymbol{s}_t, \boldsymbol{\tau}_t)); \boldsymbol{P})$ nor in $O^*(G_l; \boldsymbol{P})$. By Lemma \ref{lemma:subsetoptimal} we know that when $\tau_{\textrm{tr}}$ is chosen sufficiently large, for any $G_l \in {\cal G}_{\tau_{\textrm{tr}}}$, and $(\boldsymbol{s}, \boldsymbol{\tau})$, $O(\psi_{\boldsymbol{P}}((\boldsymbol{s}, \boldsymbol{\tau})); \boldsymbol{P})$ is a subset of $O^*(G_l; \boldsymbol{P})$. Since the solution to the AROE is a continuous function, by choosing a sufficiently large $\tau_{\textrm{tr}}$ we can control the regret due to near-optimal actions. The regret due to suboptimal actions can be bounded in the same way as in Theorem \ref{s:urestls:theorem:main}. The following theorem gives a logarithmic upper bound on the regret of AREP-FP. 

\begin{theorem} \label{thm:finite_partition}
When the true set of transition probabilities $\boldsymbol{P}$ is such that Assumption \ref{s:urestls:assump:1} is true, for a player using AREP-FP with exploration constant $L$, and threshold $\tau_{\textrm{tr}}$ sufficiently large such that for any $(\boldsymbol{s}, \boldsymbol{\tau}) \in G_l$, $G_l \in {\cal G}_{\tau_{\textrm{tr}}}$, we have $|h_{\boldsymbol{P}}(\psi_{\boldsymbol{P}}((\boldsymbol{s}, \boldsymbol{\tau}))) - h_{\boldsymbol{P}}(\psi^*(G_l; \boldsymbol{P}))|< C/2T$, where $C>0$ is a constant and $T$ is the time horizon, the regret of AREP-FP is upper bounded by
\begin{eqnarray}
 C + \left(L \log T (1+T_{\max}) +6 K S^2_{\max} \beta \right)
\times \sum_{l=1}^{A(\tau_{\textrm{tr}})} \sum_{u \notin O(J_l;\boldsymbol{P})}  \bar{\Delta}(J_l, u; \boldsymbol{P}) + \bar{C}_{\boldsymbol{P}}, \nonumber
\end{eqnarray}
for some $\delta > 0$ which depends on $L$ and $\tau_{\textrm{tr}}$.
\end{theorem}
\begin{proof} 
The regret at time $T$ is upper bounded by Lemma \ref{s:urestls:eqn:regret}. Consider any $t$ which is an exploitation step. Let $l$ be such that $(\boldsymbol{s}_t, \boldsymbol{\tau}_t) \in G_l$. If the selected arm $\alpha(t) \in O(\psi_{\boldsymbol{P}}((\boldsymbol{s}_t, \boldsymbol{\tau}_t)); \boldsymbol{P})$, then an optimal decision is made at $t$, so the contribution to regret in time step $t$ is zero. Next, we consider the case when $\alpha(t) \notin O(\psi_{\boldsymbol{P}}((\boldsymbol{s}_t, \boldsymbol{\tau}_t)); \boldsymbol{P})$. In this case there are two possibilities: either $\alpha(t) \in O^*(G_l; \boldsymbol{P})$ or not. We know that when $O^*(G_l; \hat{\boldsymbol{P}}_t) \subset O^*(G_l; \boldsymbol{P})$ we have $\alpha(t) \in O^*(G_l; \boldsymbol{P})$. Since $|h_{\boldsymbol{P}}(\psi_{\boldsymbol{P}}((\boldsymbol{s}, \boldsymbol{\tau}))) - h_{\boldsymbol{P}}(\psi^*(G_l; \boldsymbol{P}))|< C/2T$ for all $(\boldsymbol{s}, \boldsymbol{\tau}) \in G_l$, we have by (\ref{s:urestls:eqn:suboptimality}),
\begin{align}
\Delta(\psi_t ,\alpha(t) ;\boldsymbol{P}) &= {\cal L}^*(\psi_t,\boldsymbol{P}) - {\cal L}(\psi_t,\alpha(t),h_{\boldsymbol{P}},\boldsymbol{P}) \leq C/T.
\end{align}
Therefore, contribution of a near-optimal action to regret is at most $C/T$. 

Finally, consider the case when $\alpha(t) \notin O^*(G_l; \boldsymbol{P})$. This implies that either the estimated belief $\hat{\psi}_t$ is not close enough to $\psi_t$ or the estimated solution to the AROE, i.e., $\hat{h}_t$, is not close enough to $h_{\boldsymbol{P}}$. Due to the non-vanishing suboptimality gap at any belief vector $\psi^*(G_l; \boldsymbol{P})$, and since decisions of AREP-FP is only based on belief vectors corresponding to $(\boldsymbol{s}, \boldsymbol{\tau}) \in {\cal C}$, the regret due to suboptimal actions can be bounded by Theorem \ref{s:urestls:theorem:main}. We get the regret bound by combining all these results.
\end{proof} 

Note that the regret bound in Theorem \ref{thm:finite_partition} depends on $\tau_{\textrm{tr}}$ which further depends on $T$: $\tau_{\textrm{tr}}$ is chosen so that for every $G_l$ in the partition created by $\tau_{\textrm{tr}}$, the function $h_{\boldsymbol{P}}$ varies by at most $C/2T$. 
Clearly since $h_{\boldsymbol{P}}$ is a continuous function, the variation of $h_{\boldsymbol{P}}$ over $G_l$, i.e., the difference between the maximum and minimum values of $h_{\boldsymbol{P}}$ over $G_l$ decreases with the diameter of $G_l$ on the belief space. Note that there is a term in regret that depends linearly on the number of sets $A(\tau_{\textrm{tr}})$ in the partition generated by $\tau_{\textrm{tr}}$, and $A(\tau_{\textrm{tr}})$ increases proportional to $(\tau_{\textrm{tr}})^K$. This tradeoff is not taken into account in Theorem \ref{thm:finite_partition}. For example, if $(\tau_{\textrm{tr}})^K \geq T$ then the regret bound in Theorem \ref{thm:finite_partition} is of no use. Another approach is to jointly optimize the regret due to suboptimal and near-optimal actions by balancing the number of sets $A(\tau_{\textrm{tr}})$ and the variation of $h_{\boldsymbol{P}}$ on sets in ${\cal G}_{\tau_{\textrm{tr}}}$. For example, given $0<\theta \leq 1$, we can find a $\tau_{\textrm{tr}}(\theta)$ such that for any $(\boldsymbol{s}, \boldsymbol{\tau}) \in G_l$, $G_l \in {\cal G}_{\tau_{\textrm{tr}}(\theta)}$, and $C>0$, we have
\begin{align*}
|h_{\boldsymbol{P}}(\psi_{\boldsymbol{P}}((\boldsymbol{s}, \boldsymbol{\tau}))) - h_{\boldsymbol{P}}(\psi^*(G_l; \boldsymbol{P}))|< \frac{C}{2T^{\theta}}~.
\end{align*} 
Then, the regret due to near-optimal decisions will be proportional to $C T^{1-\theta}$, and the regret due to suboptimal decision will be proportional to $(\tau_{\textrm{tr}}(\theta))^K$. Let 
%
$C = \sup_{\psi \in \boldsymbol{\Psi}} h_{\boldsymbol{P}}(\psi) - \inf_{\psi \in \boldsymbol{\Psi}} h_{\boldsymbol{P}}(\psi)$.
%
Since $T^{1-\theta}$ is decreasing in $\theta$ and $\tau_{\textrm{tr}}(\theta)$ is increasing in $\theta$, there exists $\theta \in [0,1]$, such that 
%
$\theta = \argmin_{\theta' \in [0,1]} |T^{1-\theta} - (\tau_{\textrm{tr}}(\theta))^K|$,
%
%
$|T^{1-\theta} - (\tau_{\textrm{tr}}(\theta))^K| \leq (\tau_{\textrm{tr}}(\theta)+1)^K - (\tau_{\textrm{tr}}(\theta))^K$.
%
If the optimal value of $\theta$ is in $(0,1)$, then given $\theta$, the player can balance the tradeoff, and achieve sublinear regret proportional to $T^{1-\theta}$. However, since the player does not know $\boldsymbol{P}$ initially, it may not know the optimal value of $\theta$. Online learning algorithms for the player to estimate the optimal value of $\theta$ is a future research direction.

%% file: numerical.tex
\section{Numerical Examples and Comparison with Existing Online Learning Algorithms} \label{sec:num}

In this section we compare the performance of AREP with existing online learning algorithms including ABA proposed in \cite{tekin2011approx} and the online version of the myopic policy proposed in \cite{ahmad2009optimality}. 
Although AREP is computationally inefficient, we can solve the AROE approximately by using belief state quantization and relative value iteration. We modify AREP such that instead of solving the AROE at each exploitation step and computing a new estimated optimal policy, it only solves the AROE when the number of observations of transitions out of every state of every arm has increased by $C_{\textrm{inc}}$ percent compared to the last time the AROE is solved. For example, if the minimum number of observations out of any state of any arm was $N_{\textrm{old}}$ the last time AROE is solved, it will be solved again with the new estimated transition probabilities when the minimum number of observations out of any state of any arm exceeds $(1 + C_{\textrm{inc}}/100)N_{\textrm{old}}$. A similar idea is used in \cite{ortner2007logarithmic}, in which the optimal policy is recomputed whenever the confidence interval for transition probabilities is halved. 

We call time steps $t$ in which the AROE is solved by AREP as the {\em computation steps}.
The AROE is solved in the following way. A finite information-state Markov Decision Problem (MDP) is formed based on the estimated transition probabilities $\hat{\boldsymbol{P}}_t$, at a computation step $t$.
The state space of this finite MDP is
\begin{align*}
S^{\textrm{fin}} := \left\{ (\boldsymbol{s}, \boldsymbol{\tau}): s^k \in S^k, \tau^k \in \{1,\ldots, \tau_{\textrm{tr}} \}, \tau^k = 1 \textrm{ for only one } k \in {\cal K} \right\}, 
\end{align*}
and the transition probability matrix
is exactly the same as the transition probability matrix of the original information state MDP,
except when arm $k$ is in state $(s_k, \tau_{\textrm{m}})$ and not played, its state remains $(s_k, \tau_{\textrm{m}})$ instead of evolving to 
$(s_k, \tau_{\textrm{m}}+1)$. Then the optimal average reward policy for this MDP is computed using relative value iteration.
This is in some way similar to AREP-FP but there is no guarantee that this finite state approximation will yield sublinear regret. We call AREP used with this approximation method AREP with finite approximation (AREP-FA).

In this section we will show that AREP-FA achieves results that are much better than that of prior work in online learning, thus even though AREP and AREP-FP may be computationally inefficient, they are practically implementable using approximation algorithms. 
  
Our numerical experiment considers the following setting: $M=2$ (two arms) with state spaces $S^1 = S^2 = \{0,1\}$, and rewards $r^1_0 = r^2_0 =0$, $r^1_1 = r^2_1 =1$. \comc{you originally had in both places $r^1_1=r^2_2=1$; pls double check. CEM: Where? I couldn't remember.} We consider three different sets of transition probability matrices C1, C2 and C3, given in Table \ref{tab:matrices}. In C1 arms are identical and bursty, i.e., $p^k_{01} + p^k_{11} <1$, in C2 the first arm is bursty, while the second arm is not bursty, and in C3, neither arm is bursty. 
We consider the following types of algorithms. 

When the transition probabilities are known, the myopic policy \cite{ahmad2009optimality} chooses the arm with the highest one-step expected reward and is shown to be optimal in terms of the average reward when the arms have only two states and are identical in terms of their state rewards and transition probabilities, and they are bursty. We define the online version of the myopic policy as follows. Similar to AREP, the online myopic policy keeps counters $N^k_{i,j}(t)$ and $C^k_i(t)$, $i,j \in S^k$, $k \in {\cal K}$, which are used to form the transition probability estimates defined in the same way as in Section \ref{single:sec:strong-alg}. 
Whenever there is an arm $k$ and state $i \in S^k$ such that $C^k_i(t) \leq L \log t$, the player explores arm $k$. Otherwise when $C^k_i(t) > L \log t$ for all arms and states, the player chooses the arm with the highest one-step expected reward based on the information state $(\boldsymbol{s}_t, \boldsymbol{\tau}_t)$ and $\hat{\boldsymbol{P}}_t$. 

ABA is an online learning algorithm proposed in \cite{tekin2011approx}, which is a threshold variant of the approximation algorithm proposed in \cite{guha2010approximation} to solve two-state non-identical, bursty restless bandit problems. This algorithm gives an index policy which computes indices for each arm separately based only on the transition probability estimates of that arm and chooses an arm to play by comparing their indices. Given $\epsilon>0$ as an input parameter to the algorithm, it is guaranteed to achieve at least $1/(2+\epsilon)$ of the average expected reward of the optimal policy for two-state, bursty arms.
Unless otherwise stated we assume that the exploration constant $L=30$ for AREP-FA, ABA and the online myopic policy. 

We compute the total reward of AREP-FA, ABA, online myopic policy and the myopic policy with known transition probabilities for $T=20000$. Note that the myopic policy only exploits since it knows the transition probabilities and is thus used as a benchmark.
 We average our results over 100 runs of the algorithms. Average rewards of these algorithms for different transition probabilities are shown in Table \ref{tab:comparison}, when AREP-FA is run with $C_{\textrm{inc}} =5$ and ABA is run with $\epsilon=0.02$. 

We see that for C1, AREP-FA, ABA and online myopic perform roughly equally well and their total reward is very close to the myopic policy which is optimal for this case. 
\newt{It can be verified that Assumption \ref{s:urestls:assump:esuboptimality2} does not hold in this case, since arms are identical. The information state which violates this assumption is $(s,s)(\infty,\infty)$, $s \in \{0,1\}$. This information states is the center information state of a $G_l \in G_{\tau_{\textrm{tr}}}$ 
However, the information state $(s,s)(\infty,\infty)$, $s \in \{0,1\}$ that violates Assumption \ref{s:urestls:assump:esuboptimality2} will never be reached by AREP-FA, thus this assumption will not be violated during the runtime of AREP-FA.}
\comc{maybe need to elaborate in what ways it violates the assumption? as in the counter example? if so please state..} Therefore, although theoretical regret bounds requires this assumption to hold, in this example, AREP-FA performs equally well with the online myopic policy. The difference between the total rewards of the myopic policy and online learning algorithms is due to the exploration steps which are necessary to learn the unknown transition probabilities. 

For C2, since the arms are not identical the myopic policy is suboptimal. In addition, arm 2 is not bursty since $p^2_{01} + p^2_{10} =1$. However, it is almost bursty since a slight decrease in $p^2_{01}$ or $p^2_{10}$ will make it bursty. We see that AREP-FA and ABA perform nearly equally well, while the online myopic policy has a total reward $18\%$ less than AREP-FA and the myopic policy has a total reward $23\%$ less than AREP-FA. 
From the results for C1 and C2 we also see that although ABA is proved to be an approximation algorithm, its actual performance is very close to AREP-FA. It works well for C2 because arm 2 is almost bursty. 

For C3, we see that the online myopic policy performs almost as well as AREP-FA while ABA has a total reward $30\%$ less than AREP-FA. It is expected that ABA should perform poorly since both arms are not bursty, but it came as a surprise that the online myopic policy performs so well. The myopic policy achieves the highest total reward for this case since it knows the transition probabilities and does not need to explore. \newt{We would like to note that for two state arms, bursty and identical arms assumption is a sufficient condition for the optimality of the myopic policy but it is not a necessary condition. There may be other cases in which the myopic policy can also be optimal, maybe including C3. For example, in \cite{liu2013sufficient}, the authors derive sufficient conditions on the optimality of the myopic policy in both finite and infinite horizon discounted cases, both for bursty and non-bursty arms. However, the discussion on the conditions that guarantees optimality of the myopic policy is out of the scope of this paper.}
%
%
%
The results for C1-C3 shows that AREP-FA can potentially significantly outperform the other online learning algorithms.
%

\begin{table}
\vspace{-0.2in}
\centering
\setlength{\tabcolsep}{.3em}
\begin{tabular}{|l|c|c|c|}
\hline
& C1 & C2& C3  \\
\hline
$p^1_{01}$, $p^1_{10}$ &  0.2, 0.2 & 0.05, 0.05 & 0.9, 0.9 \\
\hline
$p^2_{01}$, $p^2_{10}$ & 0.2, 0.2 & 0.5, 0.5 & 0.8, 0.5 \\
\hline
\end{tabular}
\vspace{-0.1in}
\caption{State transition probabilities of the arms. Two bursty arms in C1, one bursty arm in C2, no bursty arm in C3.}
\label{tab:matrices}
\vspace{-0.2in}
\end{table}

\begin{table}
\vspace{-0.2in}
\centering
\setlength{\tabcolsep}{.3em}
\begin{tabular}{|l|c|c|c|c|}
\hline
 & AREP-FA ($C_{\textrm{inc}} = 5$) & ABA & Online myopic & Myopic  \\
\hline
C1 & 12833 & 12784 & 12809 & 13008 \\
\hline
C2 & 13262 & 13220 & 10935 & 10233 \\
\hline
C3 & 14424 & 10076 & 14408 & 14690 \\
\hline
\end{tabular}
\vspace{-0.1in}
\caption{Total rewards of AREP-FA ($C_{\textrm{inc}} = 5$), ABA, online myopic policy and myopic policy (with known transition probabilities) for three different cases C1-C3.}
\label{tab:comparison}
\vspace{-0.2in}
\end{table}
In Table \ref{tab:comp_weak} we compare the time-averaged rewards of AREP-FA with the average reward of the best arm for C1, C2 and C3. Since online learning algorithms proposed in \cite{tekin2012online} and \cite{dai2011non} have sublinear regret with respect to the best arm, and since they learn to play the best arm, their performances will be very poor compared to the performance of AREP-FA. 

\begin{table}
\vspace{-0.2in}
\centering
\setlength{\tabcolsep}{.3em}
\begin{tabular}{|l|c|c|c|}
\hline
& C1 & C2& C3  \\
\hline
AREP-FA ($C_{\textrm{inc}} = 5$) &  0.642 & 0.663 & 0.721 \\
\hline
The best arm & 0.5 & 0.5 & 0.615 \\
\hline
\end{tabular}
\vspace{-0.1in}
\caption{Comparison of the average reward of AREP-FA with the average expected reward of the best arm for cases C1-C3.}
\label{tab:comp_weak}
\vspace{-0.2in}
\end{table}

The number of time steps in which AROE is re-solved by AREP-FA by time $T$ and the total reward of AREP-FA as a function of $C_{\textrm{inc}}$ is given in Table \ref{tab:AREP-FA} for C2.
Although in theory exploration and exploitation should continue indefinitely or until the final time $T$, in order to have sublinear regret with respect to the optimal allocation, from Table \ref{tab:AREP-FA} we see that solving the AROE only when the number of observations that are used to estimate the transition probabilities multiply by some fixed amount works equally well in practice. We see that the performance does not improve as the number of computations increase. 
The reason for this can be that the structure of the optimal policy for the information states that are visited by AREP-FA in run time does not change when $\hat{\boldsymbol{P}}_t$ is slightly different from $\boldsymbol{P}$. By having sufficient exploration steps we guarantee that probability of $\hat{\boldsymbol{P}}_t$ being very different from $\boldsymbol{P}$ is very small each time the AROE is solved.

\begin{table}
\vspace{-0.2in}
\centering
\setlength{\tabcolsep}{.3em}
\begin{tabular}{|l|c|c|c|}
\hline
$C_{\textrm{inc}}$ & 5 & 10 & 50  \\
\hline
Total reward &  13262 & 13229 & 13264 \\
\hline
Average number of computations of AROE & 8.63 & 4.51 & 1.21 \\
\hline
\end{tabular}
\vspace{-0.1in}
\caption{The total reward of AREP-FA and the average number of computations of AROE as a function of $C_{\textrm{inc}}$ for C2.}
\label{tab:AREP-FA}
\vspace{-0.2in}
\end{table}

The average CPU times required to run AREP-FA, ABA and online myopic policy in MATLAB are given in Table \ref{tab:CPU} for different values of $C_{\textrm{inc}}$ for C1. The running time of ABA is the highest since it computes the approximately optimal policy at each exploitation step. From this table we see that although AREP is computationally intractable, it can be practically implemented in much the same way as ABA and the online myopic policy using approximation methods. 

\begin{table}
\vspace{-0.2in}
\centering
\setlength{\tabcolsep}{.3em}
\begin{tabular}{|l|c|c|c|c|c|}
\hline
& AREP-FA ($C_{\textrm{inc}}=5$) & ---- (10) & ---- (50) & ABA & Online myopic policy  \\
\hline
Avg run time (sec) &  2.01 & 1.26 & 0.67 & 17.2 & 0.39 \\
\hline
\end{tabular}
\vspace{-0.1in}
\caption{Average run time of AREP-FA as a function of $C_{\textrm{inc}}$, and the average run times of ABA and the online myopic policy in MATLAB for C1.}
\label{tab:CPU}
\vspace{-0.2in}
\end{table}

Finally, we give the performance of AREP-FA as a function of the exploration constant $L$. Recall that for logarithmic regret bound to hold, $L$ should be chosen large enough for AREP. If no bound on $L$ is known, $L$ should be chosen increasingly over time such that our bounds will hold after some time. However, this increases regret since the number of explorations increases over time. The average reward of AREP-FA as a function of $L$ is given in Table \ref{tab:D} for C2. We see that choosing $L$ very small has a much less negative impact than choosing $L$ very large. This is due to the fact that the transition probabilities can still be learned over exploitation steps, and the loss of reward due to the suboptimal decisions made on exploitation steps as a result of the poorly estimated transition probabilities is not much greater than the loss incurred due to explorations. On the countrary, when $L$ is very large, e.g., $L=300$, there are 13341 exploration steps on average, thus losses from explorations is much more than the gains in exploitation steps. 

\begin{table}
\vspace{-0.2in}
\centering
\setlength{\tabcolsep}{.3em}
\begin{tabular}{|l|c|c|c|c|}
\hline
Exploration constant $L$ & 0.3 & 3 & 30 & 300 \\
\hline
Total reward ($C_{\textrm{inc}}=5)$ & 12264 & 12919 & 13262 & 11463 \\
\hline
\end{tabular}
\vspace{-0.1in}
\caption{The total reward of AREP-FA as a function of $L$ for C2.}
\label{tab:D}
\vspace{-0.2in}
\end{table}

%% file: s_urestls_app_lemmalargedev1.tex
Certain results from the large deviation theory are frequently used, e.g., to relate the accuracy of the player's transition probability estimates to its probability of deviating from the optimal action. We begin with the definition of a uniformly ergodic Markov chain.
\begin{definition} \label{intro:defn:ergodic}
\cite{mitrophanov2005} A Markov chain $X=\{X_t, t \in \mathbb{Z}_+ \}$ on a measurable space $({\cal S}, {\cal B})$, with transition kernel $P(x, {\cal G})$ is uniformly ergodic if there exists constants $\rho<1, C<\infty$ such that for all $x \in {\cal S}$,
\begin{align}
\left\| e_x P^t - \pi \right\| \leq C \rho^t, t \in \mathbb{Z}_+ ~, \label{eqn:ergodicity}
\end{align}
\end{definition}
where $e_x$ is the unit vector indicating that the initial state is $x$, and the total variation norm is used. For finite and countable vectors this corresponds to $l_1$ norm, and the induced matrix norm corresponds to the maximum absolute row sum norm. Clearly, for a finite state Markov chain uniform ergodicity is equivalent to ergodicity. The next is a bound on a perturbation to a uniformly ergodic Markov chain. 

\begin{lemma} \label{intro:lemma:ergodic}
(\cite{mitrophanov2005} Theorem 3.1.) Let $X=\{X_t, t \in \mathbb{Z}_+ \}$ be a uniformly ergodic Markov chain for which (\ref{eqn:ergodicity}) holds. Let $\hat{X}=\{\hat{X}_t, t \in \mathbb{Z}_+ \}$ be the perturbed chain with transition kernel $\hat{P}$. Given the two chains have the same initial distribution, let $\psi_t, \hat{\psi}_t$ be the distribution of $X, \hat{X}$ at time $t$, respectively. Then,
\begin{align}
\left\|\psi_t - \hat{\psi}_t \right\| \leq C_1(P,t) \left\|\hat{P}-P \right\|, \label{eqn:perturbation}
\end{align}
where $C_1(P,t) = \left(\hat{t} + C \frac{\rho^{\hat{t}}-\rho^t}{1-\rho}\right)$ and $\hat{t} = \left\lceil \log_\rho C^{-1} \right\rceil$.
\end{lemma}

Next, the Chernoff-Hoeffding bound is frequently used in our proofs that bounds the difference between the sample mean and the expected reward on distributions with bounded support.
\begin{lemma} \label{intro:lemma:chernoff}
(Chernoff-Hoeffding Bound) Let $X_1,\ldots,X_T$ be random variables with common range [0,1], such that $E[X_t|X_{t-1},\ldots,X_1]=\mu$ for $t=2, 3, \ldots, T$. Let $S_T=X_1+\ldots+X_T$. Then for all $\epsilon \geq 0$
\begin{align*}
P(|S_T - T\mu| \geq \epsilon) \leq 2 e^{\frac{-2\epsilon^2}{T}}.
\end{align*}
\end{lemma}

The following lemma is used to relate the estimate estimated belief state of the player with the true belief state; it gives an upper bound on the difference between the product of two equal-sized sets of numbers in the unit interval, in terms of the sum of the absolute values of the pairwise differences between the numbers taken from each set.
\begin{lemma} \label{s:urestls:lemma:sumbound}
for $\rho_k, \rho_{k}' \in [0,1]$ we have
\begin{align}
|\rho_1 \ldots \rho_K - \rho_{1}' \ldots \rho_{K}'| \leq \sum_{k=1}^K |\rho_k - \rho_{k}'| ~. \label{s:urestls:eqn:sumbound}
\end{align}
\end{lemma}
\begin{proof}
First consider $|\rho_1 \rho_2 - \rho_1' \rho_2'|$ where $\rho_1, \rho_2, \rho_1', \rho_2' \in [0,1]$. Let $\epsilon=\rho_2'-\rho_2$. Then
\begin{align*}
|\rho_1 \rho_2 - \rho_1' \rho_2'| = |\rho_1 \rho_2 - \rho_1'(\rho_2+\epsilon)|
= |\rho_2 (\rho_1 - \rho_1') - \rho_1' \epsilon|
\leq \rho_2 |\rho_1-\rho_1'|+ \rho_1'|\epsilon|. \notag
\end{align*}
But we have
\begin{align}
|\rho_1 - \rho_1'| + |\rho_2 - \rho_2'| &= |\rho_1 - \rho_1'| + |\epsilon|
\geq \rho_2 |\rho_1-\rho_1'|+ \rho_1'|\epsilon|. \notag
\end{align}
Thus
\begin{align}
|\rho_1 \rho_2 - \rho_1' \rho_2'| \leq |\rho_1 - \rho_1'| + |\rho_2 - \rho_2'|. \notag
\end{align}
We now use induction. Clearly (\ref{s:urestls:eqn:sumbound}) holds for $K=1$. Assume it holds for some $K>1$. Then
\begin{align}
|\rho_1 \ldots \rho_{K+1} - \rho_1' \ldots \rho_{K+1}'| &\leq |\rho_1 \ldots \rho_{K} - \rho_1' \ldots \rho_{K}'| + |\rho_{K+1} - \rho_{K+1}'|
\leq \sum_{k=1}^K |\rho_k - \rho_{k}'|. \notag
\end{align}
\end{proof}

%% file: s_urestls_app_lemmalargedev.tex
\aug{
Let $L \geq 1/(\epsilon'^2)$.
Let $A^k_{i,t}$ be the random variable which is the next state observed after state $i$ of arm $k$, for the $t$th time the player selects arm $k$ immediately after it had selected arm $k$ and observed state $i$.
Since conditional on being in state $i$, the next state of arm $k$ is drawn from distribution $P^k(\cdot|i)$, the random variables $A^k_{i,1}, A^k_{i,2}, \ldots, A^k_{i,t}$ are i.i.d. with distribution $P^k(\cdot|i)$. 
We have
\begin{align}
\hat{p}^k_{ij,t} = \frac{\sum_{t'=1}^{C^k_i(t)} I \left(  A^k_{i,t'} =j \right)} {C^k_i(t)}  .    \notag
\end{align}
Then using Lemma \ref{intro:lemma:chernoff}, we have
\begin{align}
P\left(|\hat{p}^k_{ij,t} - p^k_{ij}| > \epsilon', {\cal E}_t \right) & = P\left(  \left| \sum_{t'=1}^{C^k_i(t)} I \left(  A^k_{i,t'} =j \right)/ C^k_i(t)  - p^k_{ij}   \right| > \epsilon', {\cal E}_t   \right) \notag \\ 
& = \sum_{C \geq 0} P \left( \left. \left| \frac{\sum_{t'=1}^{C^k_i(t)} I(A^k_i(t') = j) }{C^k_i(t)}     - p^k_{ij} \right| > \epsilon'  ,  {\cal E}_t      \right\vert C^k_i(t) = C  \right) P(C^k_i(t) = C)    \notag \\
& = \sum_{C \geq L \log t}  P \left( \left. \left| \frac{\sum_{t'=1}^{C^k_i(t)} I(A^k_i(t') = j) }{C^k_i(t)}     - p^k_{ij} \right| > \epsilon', {\cal E}_t    \right\vert C^k_i(t) = C  \right) P(C^k_i(t) = C) \notag \\
&\leq \sum_{C \geq L \log t}  P \left( \left. \left| \frac{\sum_{t'=1}^{C} I(A^k_i(t') = j) }{C}     - p^k_{ij} \right| > \epsilon'   \right\vert C^k_i(t) = C  \right) P(C^k_i(t) = C) \notag \\ 
&\leq   \sum_{C \geq L \log t}  2 e^{-2 L \log t \epsilon'^2} P(C^k_i(t) = C) \notag \\
&\leq 2 e^{-2 L \log t \epsilon'^2}  \leq 2/t^2 , \notag
\end{align}
since $C^k_i(t) \geq L \log t$ for all $k \in {\cal K}$, $i \in S^k$ if and only if ${\cal E}_t$ happens.
}

%% file: s_urestls_app_lemma_subsec30.tex
When the estimated belief is in $J_l$, for any suboptimal action $u$, we have 
\begin{align}
{\cal L}^*(\psi_t, \boldsymbol{P}) - {\cal L}(\psi_t, u, h_{\boldsymbol{P}}, \boldsymbol{P}) \geq \underline{\Delta}. \label{strong:eqn:ineqaroe1}
\end{align}
Let $\epsilon < \underline{\Delta}/4$. When ${\cal F}_t(\epsilon)$ occurs, we have
\begin{align}
\left| {\cal I}_t(\hat{\psi}_t,u) - {\cal L}(\hat{\psi}_t, u, h_{\boldsymbol{P}}, \hat{\boldsymbol{P}}_t ) \right| \leq \epsilon, \label{strong:eqn:ineqaroe3}
\end{align}
for all $u \in U$.
Since $T_{\boldsymbol{P}}(\psi,y,u)$ is continuous in $\boldsymbol{P}$, and $h_{\boldsymbol{P}}(\psi)$ is continuous in $\psi$, there exists $\delta_e >0$
such that $|| \hat{\boldsymbol{P}}_t - \boldsymbol{P}||_1 < \delta_e$ implies that
\begin{align}
\left| {\cal L}(\hat{\psi}_t, u, h_{\boldsymbol{P}}, \boldsymbol{P} ) - {\cal L}(\hat{\psi}_t, u, h_{\boldsymbol{P}}, \hat{\boldsymbol{P}}_t ) \right| \leq \underline{\Delta}/4, \label{strong:eqn:ineqaroe4}
\end{align}
for all $u \in U$. 
Let $u^* \in O(J_l; \boldsymbol{P})$. Using (\ref{strong:eqn:ineqaroe1}), (\ref{strong:eqn:ineqaroe3}) and (\ref{strong:eqn:ineqaroe4}), we have
\begin{align}
{\cal I}(\hat{\psi}_t, u^*) &\geq {\cal L}(\hat{\psi}_t, u^*, h_{\boldsymbol{P}}, \hat{\boldsymbol{P}}_t ) -\epsilon \notag \\
&\geq {\cal L}(\hat{\psi}_t, u^*, h_{\boldsymbol{P}}, \boldsymbol{P} ) -\epsilon - \underline{\Delta}/4 \notag \\
&= {\cal L}^*(\hat{\psi}_t, \boldsymbol{P}) -\epsilon - \underline{\Delta}/4 \notag \\
&\geq {\cal L}(\hat{\psi}_t, u, h_{\boldsymbol{P}}, \boldsymbol{P}) + 3 \underline{\Delta}/4 -\epsilon \notag \\
&\geq {\cal L}(\hat{\psi}_t, u, h_{\boldsymbol{P}}, \hat{\boldsymbol{P}}_t ) + \underline{\Delta}/2 -\epsilon \notag \\
&\geq {\cal I}(\hat{\psi}_t, u) + \underline{\Delta}/2 -2\epsilon \notag \\
&> {\cal I}(\hat{\psi}_t, u) ~. \notag
\end{align}
Therefore, we have
\begin{align}
\left\{\hat{\psi}_t \in J_l, U_t=u, || \hat{\boldsymbol{P}}_t - \boldsymbol{P}||_1 < \delta_e, {\cal E}_t, {\cal F}_t \right\} = \emptyset ~. \label{strong:eqn:ineqaroe5}
\end{align}
Recall that for any $u \notin O(J_l; \boldsymbol{P})$,
\begin{align}
E^{\boldsymbol{P}}_{\psi_0,\alpha}[ D_{1,1}(T,\epsilon,J_l,u)] 
&= \sum_{t=1}^{T} P \left( \hat{\psi}_t \in J_l, U_t=u, {\cal E}_t, {\cal F}_t \right) \notag \\
&= \sum_{t=1}^{T} P \left( \hat{\psi}_t \in J_l, U_t=u, || \hat{\boldsymbol{P}}_t - \boldsymbol{P}||_1 < \delta_e, {\cal E}_t, {\cal F}_t \right) \notag \\
&+ \sum_{t=1}^{T} P \left( \hat{\psi}_t \in J_l, U_t=u, || \hat{\boldsymbol{P}}_t - \boldsymbol{P}||_1 \geq \delta_e, {\cal E}_t, {\cal F}_t \right) \notag \\
&\leq \sum_{t=1}^{T} P \left( ||\hat{\boldsymbol{P}}_t - \boldsymbol{P}||_1 \geq \delta_e, {\cal E}_t \right), \label{strong:eqn:ineqaroe6}
\end{align}
where (\ref{strong:eqn:ineqaroe6}) follows from (\ref{strong:eqn:ineqaroe5}).
Therefore for any $u \notin O(J_l; \boldsymbol{P})$,
\begin{align}
E^{\boldsymbol{P}}_{\psi_0,\alpha} \left[D_{1,1}(T,\epsilon,J_l,u)\right] &\leq \sum_{t=1}^{T} P\left(\left\|\hat{\boldsymbol{P}}_t - \boldsymbol{P}\right\|_1 \geq \delta_e, {\cal E}_t\right) \notag \\
&\leq \sum_{t=1}^{T} P\left( \left\{ |\hat{p}^k_{ij,t} - p^k_{ij}| \geq \frac{\delta_e}{K S^2_{\max}}, \textrm{ for some } k \in {\cal K},~ i,j \in S^k \right\}, {\cal E}_t \right) \notag \\
&\leq \sum_{t=1}^{T} \sum_{k=1}^K \sum_{(i,j) \in S^k \times S^k} P\left(|\hat{p}^k_{ij,t} - p^k_{ij}| \geq \frac{\delta_e}{K S^2_{\max}}, {\cal E}_t\right) \notag \\
&\leq 2 K S^2_{\max} \beta, \notag
\end{align}
for \aug{$L \geq K^2 S^4_{\max}/ \delta_e^2$}, where the last inequality follows from Lemma \ref{s:urestls:lemma:largedev}.

%% file: s_urestls_app_lemmasubsec34.tex
Since $h_{\tilde{\boldsymbol{P}}}$ is continuous in $\psi$ by Lemma \ref{s:urestls:lemma:ACOEexistence} for any $\tilde{\boldsymbol{P}}$ such that Assumption \ref{s:urestls:assump:1} holds, and since $\bar{r}(\psi), V_{\tilde{\boldsymbol{P}}}, T_{\tilde{\boldsymbol{P}}}$ are continuous in $\tilde{\boldsymbol{P}}$, we have for any $\psi \in \boldsymbol{\Psi}$:
\begin{align}
&g_{\hat{\boldsymbol{P}}} + h_{\hat{\boldsymbol{P}}}(\psi)
= \argmax_{u \in U} \left\{  \bar{r}(\psi, u)
+ \sum_{y \in S^u} V_{\hat{\boldsymbol{P}}}(\psi,y, u) 
h_{\hat{\boldsymbol{P}}} (T_{\hat{\boldsymbol{P}}}(\psi, y, u) )  \right\}  \notag \\
& = \argmax_{u \in U} \left\{  \bar{r}(\psi, u) 
+ \sum_{y \in S^u} V_{\boldsymbol{P}}(\psi, y, u) 
h_{\hat{\boldsymbol{P}}} (T_{\boldsymbol{P}}(\psi, y, u) )
+ q (\boldsymbol{P}, \hat{\boldsymbol{P}}, \psi, u)
\right\}, \label{eqn:perturbedAROE1}
\end{align}
for some function $q$ such that 
%
$\lim_{\hat{\boldsymbol{P}} \rightarrow \boldsymbol{P}} q (\boldsymbol{P}, \hat{\boldsymbol{P}}, \psi, u) = 0, ~~\forall \psi \in \boldsymbol{\Psi}, u \in U$.
%
Let $\bar{r}(\boldsymbol{P}, \hat{\boldsymbol{P}}, \psi, u) = \bar{r}(\psi, u) + q (\boldsymbol{P}, \hat{\boldsymbol{P}}, \psi, u)$. We can write (\ref{eqn:perturbedAROE1}) as 

\begin{align}
g_{\hat{\boldsymbol{P}}} + h_{\hat{\boldsymbol{P}}}(\psi) 
= \argmax_{u \in U} \left\{ \bar{r}(\boldsymbol{P}, \hat{\boldsymbol{P}}, \psi, u)
+ \sum_{y \in S^u} V_{\boldsymbol{P}}(\psi, y, u) 
h_{\hat{\boldsymbol{P}}} (T_{\boldsymbol{P}}(\psi, y, u) )
  \right\}~. \label{eqn:perturbedAROE2}
\end{align}
Note that (\ref{eqn:perturbedAROE2}) is the AROE for a system with set of transition probability matrices $\boldsymbol{P}$, and perturbed rewards $ \bar{r}(\boldsymbol{P}, \hat{\boldsymbol{P}}, \psi, u)$. Since 
%
$\lim_{\hat{\boldsymbol{P}} \rightarrow \boldsymbol{P}} r (\boldsymbol{P}, \hat{\boldsymbol{P}}, \psi, u) = \bar{r}(\psi, u), ~~\forall \psi \in \boldsymbol{\Psi}, u \in U$,
%
we expect $h_{\hat{\boldsymbol{P}}}$ to converge to $h_{\boldsymbol{P}}$.  Below we prove that this is true. Let $F_{\hat{P}}$ denote the dynamic programming operator defined in (\ref{s:urestls:eqn:DPOP}), with transition probabilities $\boldsymbol{P}$ and rewards $r(\boldsymbol{P}, \hat{\boldsymbol{P}}, \psi, u)$. Then, by S-1 of Lemma \ref{s:urestls:lemma:ACOEexistence}, there exists a sequence of functions $v_{0, \hat{\boldsymbol{P}}}, v_{1, \hat{\boldsymbol{P}}} , v_{2, \hat{\boldsymbol{P}}}, \ldots$ such that $v_{0, \hat{\boldsymbol{P}}} = 0$, $ v_{l, \hat{\boldsymbol{P}}} = F_{\hat{\boldsymbol{P}}} v_{l-1, \hat{\boldsymbol{P}}}$ and another sequence of functions $v_{0, \boldsymbol{P}}, v_{1, \boldsymbol{P}} , v_{2, \boldsymbol{P}}, \ldots$ such that $v_{0, \boldsymbol{P}} = 0$, $ v_{l, \boldsymbol{P}} = F_{\boldsymbol{P}} v_{l-1, \boldsymbol{P}}$, for which
\begin{align}
\lim_{l \rightarrow \infty} v_{l, \hat{\boldsymbol{P}}} = h_{\hat{\boldsymbol{P}}}, \label{eqn:perturbediteration} \\
\lim_{l \rightarrow \infty} v_{l, \boldsymbol{P}} = h_{\boldsymbol{P}}, \label{eqn:iteration} 
\end{align}
uniformly in $\psi$.
Let
\begin{align*}
q_{\max}(\boldsymbol{P}, \hat{\boldsymbol{P}}) := \max_{u \in U, \psi \in \boldsymbol{\Psi}} |q (\boldsymbol{P}, \hat{\boldsymbol{P}}, \psi, u)| ~.
\end{align*}
We have
\begin{align*}
 v_{1, \boldsymbol{P}}(\psi) &= \max_{u \in U} \left\{  \bar{r}(\psi, u)   \right\} \\
 v_{1, \hat{\boldsymbol{P}}}(\psi) &= \max_{u \in U} \left\{  \bar{r}(\psi, u) + q (\boldsymbol{P}, \hat{\boldsymbol{P}}, \psi, u)   \right\}.
\end{align*}
\aug{
Next, we prove that 
\begin{align}
|v_{l, \boldsymbol{P}}(\psi) -  v_{l, \hat{\boldsymbol{P}}}(\psi)|   \leq  l q_{\max}(\boldsymbol{P}, \hat{\boldsymbol{P}})       \label{eqn:inductionproof}
\end{align}
for all $\psi$ by induction.
Clearly, we have for all $\psi \in \boldsymbol{\Psi}$
\begin{align}
| v_{1, \boldsymbol{P}}(\psi) -  v_{1, \hat{\boldsymbol{P}}}(\psi)|   \leq  q_{\max}(\boldsymbol{P}, \hat{\boldsymbol{P}}) .  \notag
\end{align}
Let 
\begin{align}
 \phi_{l,\hat{\boldsymbol{P}}}(\psi,u) :=   \bar{r}(\psi, u) + \sum_{ y \in S^u } V_{\boldsymbol{P}}(\psi,y,u) v_{l-1,\hat{\boldsymbol{P}}} (T_{\boldsymbol{P}}(\psi,y,u)) .   \notag
\end{align}
Assume that for all $\psi \in \boldsymbol{\Psi}$, 
\begin{align}
| v_{l, \boldsymbol{P}}(\psi) -  v_{l, \hat{\boldsymbol{P}}}(\psi)|   \leq  l q_{\max}(\boldsymbol{P}, \hat{\boldsymbol{P}})   .  \notag
\end{align}
This implies that for all $u \in {\cal U}$
\begin{align}
|  \phi_{l+1,\hat{\boldsymbol{P}}}(\psi,u)   -  \phi_{l+1,\boldsymbol{P}}(\psi,u) | 
& \leq q_{\max}(\boldsymbol{P}, \hat{\boldsymbol{P}})  + \sum_{ y \in S^u } V_{\boldsymbol{P}}(\psi,y,u) 
| v_{l, \boldsymbol{P}}(\psi) -  v_{l, \hat{\boldsymbol{P}}}(\psi)|    \notag \\
&\leq (l+1) q_{\max}(\boldsymbol{P}, \hat{\boldsymbol{P}}) \notag,
\end{align}
hence,
\begin{align}
| v_{l+1, \boldsymbol{P}}(\psi) -  v_{l+1, \hat{\boldsymbol{P}}}(\psi)|  
 & = |\max_{u \in U} \{  \phi_{l+1,\boldsymbol{P}}(\psi,u) \}   - \max_{u \in U } \{  \phi_{l+1,\hat{\boldsymbol{P}}}(\psi,u) \}|
\leq  (l+1) q_{\max}(\boldsymbol{P}, \hat{\boldsymbol{P}}) .     \notag
\end{align}

Fix an $\epsilon>0$. Let $B_{\eta_1(\epsilon)}(\boldsymbol{P})$ be the compact ball with radius $\eta_1(\epsilon)$ centered at $\boldsymbol{P}$ for which Assumption 1 holds for every $\hat{\boldsymbol{P}} \in B_{\eta_1(\epsilon)}(\boldsymbol{P})$  (Since Assumption 1 holds for $\boldsymbol{P}$, existence of such a compact ball is guaranteed). 
From (\ref{eqn:perturbediteration}) and (\ref{eqn:iteration}) for any $\epsilon>0$, there exists $\eta_1(\epsilon)>0$ such that for all $\hat{\boldsymbol{P}} \in B_{\eta_1(\epsilon)}(\boldsymbol{P})$ there exists an integer
$N_1(\hat{\boldsymbol{P}} )$ such that for all $l > N_1(\hat{\boldsymbol{P}} )$, we have
\begin{align}
|v_{l,\hat{\boldsymbol{P}}}(\psi) - h_{\hat{\boldsymbol{P}}} (\psi)    |   & \leq \epsilon/3 .   \notag 
\end{align}
Let $l^*$ be the smallest integer that is greater than $\max_{\hat{\boldsymbol{P}} \in B(\eta_1(\epsilon)) }N_1(\hat{\boldsymbol{P}} )$. 
Since $\lim_{\hat{\boldsymbol{P}} \rightarrow \boldsymbol{P}} q_{\max}(\boldsymbol{P}, \hat{\boldsymbol{P}})  = 0 $, there exists $\eta_2(\epsilon) < \eta_1(\epsilon)$ such that for all $\hat{\boldsymbol{P}} \in B_{\eta_2(\epsilon)}(\boldsymbol{P})$ we have
\begin{align}
q_{\max}(\boldsymbol{P}, \hat{\boldsymbol{P}}) \leq \epsilon/(3 l^*) ,    \notag
\end{align}
which implies from (\ref{eqn:inductionproof}) that 
\begin{align}
| v_{l^*, \boldsymbol{P}}(\psi) -  v_{l^*, \hat{\boldsymbol{P}}}(\psi)|   \leq \epsilon/3 .    \notag
\end{align}
This implies that for any $\epsilon>0$, there exists $\eta_2(\epsilon) >0$ such that for all $\hat{\boldsymbol{P}} \in B_{\eta_2(\epsilon)}(\boldsymbol{P})$, $\psi \in \boldsymbol{\Psi}$ we have
\begin{align}
| h_{\boldsymbol{P}}(\psi) -  h_{\hat{\boldsymbol{P}}}(\psi)|
\leq | h_{\boldsymbol{P}}(\psi) - v_{l^*, \boldsymbol{P}}(\psi) | 
+ |v_{l^*, \hat{\boldsymbol{P}}}(\psi) - v_{l^*, \boldsymbol{P}}(\psi)| + |v_{l^*, \hat{\boldsymbol{P}}}(\psi) -  h_{\hat{\boldsymbol{P}}}(\psi)| 
 < \epsilon, \notag
\end{align}
To complete the proof, let $\varsigma(\epsilon) >0$ be the largest number such that $||P^k - \hat{P}^k||_1 < \varsigma(\epsilon) $, for all $k \in {\cal K}$ implies that $\hat{\boldsymbol{P}} \in B_{\eta_2(\epsilon)}(\boldsymbol{P})$. 
}

%% file: s_urestls_app_lemma_subsec35.tex
Let $\varsigma = \varsigma(\epsilon)$, given in Lemma \ref{s:urestls:lemma:subsec34}.  We have by Lemma \ref{s:urestls:lemma:subsec34},
\begin{align*}
\left\{ \left\|P^k - \hat{P}^k_t\right\|_1 < \varsigma, \forall k \in {\cal K}  \right\} \subset \left\{ \left\|h_{\boldsymbol{P}} - h_t \right\|_\infty < \epsilon \right\} ~,
\end{align*}
which implies
\begin{align*}
\left\{ \left\|P^k - \hat{P}^k_t\right\|_1 \geq \varsigma, \textrm{ for some } k \in {\cal K} \right\} \supset \left\{ \left\|h_{\boldsymbol{P}} - h_t \right\|_\infty \geq \epsilon \right\} ~.
\end{align*}
Then 
\begin{align}
 E^{\boldsymbol{P}}_{\psi_0,\alpha} [D_{1,2}(T,\epsilon)] &= E^{\boldsymbol{P}}_{\psi_0,\alpha} \left[\sum_{t=1}^{T} I({\cal E}_t, {\cal F}^c_t(\epsilon))\right] \notag \\
&\leq \sum_{t=1}^{T} P \left(\left\|P^k - \hat{P}^k_t\right\|_1 \geq \varsigma, \textrm{ for some } k \in {\cal K}, {\cal E}_t \right) \notag \\
&\leq \sum_{k=1}^K \sum_{(i,j) \in S^k \times S^k} \sum_{t=1}^{T} P\left(|p^k_{ij} - \hat{p}^k_{ij,t}| > \frac{\varsigma}{S^2_{\max}}, {\cal E}_t \right) \notag \\
&\leq 2 K S^2_{\max} \beta ~, \notag
\end{align}
for $L \geq S^4_{\max}/ \varsigma^2$.

%% file: s_urestls_app_lemma_subsec21.tex
Consider $t>0$
\begin{align}
| (\hat{\psi}_t)_{\boldsymbol{x}} - (\psi_t)_{\boldsymbol{x}} | &= \left| \prod_{k=1}^K \left((\hat{P}^k_t)^{\tau^k} e^k_{s^k} \right)_{x^k} - \prod_{k=1}^K \left((P^k)^{\tau_k} e^k_{s^k} \right)_{x^k} \right| \notag \\
&\leq \sum_{k=1}^K \left|\left((\hat{P}^k_t)^{\tau^k} e^k_{s^k} \right)_{x^k}-\left((P^k)^{\tau_k} e^k_{s^k} \right)_{x^k} \right| \notag \\
&\leq \sum_{k=1}^K \left\| (\hat{P}^k_t)^{\tau^k} e^k_{s^k} - (P^k)^{\tau^k} e^k_{s^k}  \right\|_1 \notag \\
&\leq C_1(\boldsymbol{P}) \sum_{k=1}^K \left\| \hat{P}^k_t - P^k \right\|_1, \label{s:urestls:eqn:subsec21}
\end{align}
where last inequality follows from Lemma \ref{intro:lemma:ergodic}. By (\ref{s:urestls:eqn:subsec21})
\begin{align}
\left\| \hat{\psi}_t - \psi_t \right\|_1 \leq |S^1| \ldots |S^K| C_1(\boldsymbol{P}) \sum_{k=1}^K \left\| \hat{P}^k_t - P^k \right\|_1. \notag
\end{align}
Thus we have
\begin{align}
& P\left( \left\| \hat{\psi}_t - \psi_t \right\|_1 > \epsilon, {\cal E}_t \right) \notag \\  
&\leq  P\left(\sum_{k=1}^K \left\| \hat{P}^k_t - P^k \right\|_1 > \epsilon/(|S^1| \ldots |S^K| C_1(\boldsymbol{P})), {\cal E}_t \right) \notag \\
&\leq \sum_{k=1}^K P\left(\left\| \hat{P}^k_t - P^k \right\|_1 > \epsilon/(K|S^1| \ldots |S^K| C_1(\boldsymbol{P})), {\cal E}_t \right) \notag \\
&\leq \sum_{k=1}^K \sum_{(i,j) \in S^k \times S^k} P\left(|\hat{p}^k_{ij,t} - p^k_{ij}| > 
\frac{\epsilon}{(K S^2_{\max} |S^1| \ldots |S^K| C_1(\boldsymbol{P}))}, {\cal E}_t \right) \notag \\
&\leq 2 K S^2_{\max} \frac{1}{t^2}, \notag
\end{align}
where last inequality follows from Lemma \ref{s:urestls:lemma:largedev} since \aug{$L \geq (K S^2_{\max} |S^1| \ldots |S^K| C_1(\boldsymbol{P}))^2/ \epsilon^2$}. Then,
\begin{align*}
E^{\boldsymbol{P}}_{\psi_0,\alpha}[D_{2,1}(T,\epsilon)] &= \sum_{t=1}^{T} P_{\psi_0,\alpha} \left(\left\| \psi_t-\hat{\psi}_t \right\|_1 > \epsilon, {\cal E}_t \right) 
\leq  2 K S^2_{\max} \beta.
\end{align*}